\documentclass[11pt,reqno]{amsart}
\usepackage{amsmath,amssymb,amsthm,enumerate,natbib,color,bbm,ifthen,graphicx,overpic}
\usepackage[latin1]{inputenc}

\def\rset{\mathbb R}

\newcommand{\eqdef}{\ensuremath{\stackrel{\mathrm{def}}{=}}}

\def\Xset{\mathsf{X}} 
\def\Vset{\mathsf{V}} 
\def\Yset{\mathsf{Y}} 
\def\F{\mathcal{F}} 
\def\B{\mathcal{B}} 
\def\e{\mathcal{E}}
\def\dist{\textsf{d}}

\def\S{\mathcal{S}}

\def\M{\mathcal{M}}

\def\G{\mathbb{G}}
\def\A{\mathcal{A}}

\def\PP{\mathbb{P}} 
\def\PE{\mathbb{E}} 
\def\cPP{\check{\mathbb{P}}} 
\def\cPE{\check{\mathbb{E}}} 

\def\L{\mathcal{L}} 

\def\compact{\mathsf{K}}  

\def\deg{\textsf{deg}}

\usepackage{bm}

\newlength{\noteWidth}
\newtheorem{theo}{Theorem}[section]
\newtheorem{lemma}[theo]{Lemma}
\newtheorem{coro}[theo]{Corollary}
\newtheorem{prop}[theo]{Proposition}

\theoremstyle{remark}
\newtheorem{rem}{Remark}
\newtheorem{example}{Example}

\newcounter{hypoconbis}
\newcounter{saveconbis}
\newcommand\debutA{\begin{list} {\textbf{A\arabic{hypoconbis}}}{\usecounter{hypoconbis}}\setcounter{hypoconbis}{\value{saveconbis}}}
\newcommand\finA{\end{list}\setcounter{saveconbis}{\value{hypoconbis}}}

\newcounter{hypoconbisp}
\newcounter{saveconbisp}
\newcommand\debutAp{\begin{list} {\textbf{A\arabic{hypoconbisp}'}}{\usecounter{hypoconbisp}}\setcounter{hypoconbisp}{\value{saveconbisp}}}
\newcommand\finAp{\end{list}\setcounter{saveconbisp}{\value{hypoconbisp}}}

\newcounter{hypocom}
\newcounter{savecom}
\newcommand{\debutB}{\begin{list}{\textbf{B\arabic{hypocom}}}{\usecounter{hypocom}}\setcounter{hypocom}{\value{savecom}}}
\newcommand{\finB}{\end{list}\setcounter{savecom}{\value{hypocom}}}

\newcounter{hypocomp}
\newcounter{savecomp}
\newcommand{\debutBp}{\begin{list}{\textbf{B\arabic{hypocomp}'}}{\usecounter{hypocomp}}\setcounter{hypocomp}{\value{savecomp}}}
\newcommand{\finBp}{\end{list}\setcounter{savecomp}{\value{hypocomp}}}

\newcounter{hypostab}
\newcounter{savestab}
\newcommand{\debutC}{\begin{list}{\textbf{C\arabic{hypostab}}}{\usecounter{hypostab}}\setcounter{hypostab}{\value{savestab}}}
\newcommand{\finC}{\end{list}\setcounter{savestab}{\value{hypostab}}}

\newcounter{hypodist}
\newcounter{savedist}
\newcommand{\debutD}{\begin{list}{\textbf{D\arabic{hypodist}}}{\usecounter{hypodist}}\setcounter{hypodist}{\value{savedist}}}
\newcommand{\finD}{\end{list}\setcounter{savedist}{\value{hypodist}}}

\begin{document}

\title[Statistical estimation of network structures]{Estimation of network structures from partially observed Markov random fields}

\author{Yves F. Atchad\'e}  \thanks{ Y. F. Atchad\'e: University of Michigan, 1085 South University, Ann Arbor,
  48109, MI, United States. {\em E-mail address:} yvesa@umich.edu}



\maketitle

\begin{center} (Aug. 2011) \end{center}

\begin{abstract}
We consider the estimation of high-dimensional network structures  from partially observed Markov random field data  using  a penalized pseudo-likelihood approach. We fit a misspecified model obtained by ignoring the missing data problem. We study the consistency of the estimator and derive a bound on its rate of convergence. The results obtained relate the rate of convergence of the estimator to the  extent of the missing data problem.   We report some simulation results that empirically validate some of the theoretical findings. 
\end{abstract}

\bigskip

\setcounter{secnumdepth}{3}

\section{Introduction and statement of the results}\label{intro}
The problem of high-dimensional network structure estimation has recently attracted a lot of attention in statistics and machine learning.  
Both in the continuous case using Gaussian graphical models (\cite{drtonetperlman04,meinshausen06,yuanetlin07,daspremontetal08,bickeletlevina08b,rothmanetal08,lametfan09}), and in the discrete case using Markov random fields (\cite{barnejeeetal08,hoefling09,ravikumaretal10,guoetal10}).  This paper focuses mainly on Markov Random Fields (MRF) for non-Gaussian data. The problem can be described as follows. Let $(X^{(1)},\ldots,X^{(n)})$ be $n$ i.i.d. random variables where $X^{(i)}=(X_1^{(i)},\ldots,X^{(i)}_{p})$ is a $p$-dimensional vector of dependent random variables with joint density
\begin{equation}\label{model}
f_\theta(x_1,\ldots,x_{p})=\frac{1}{Z_\theta}\exp\left\{\sum_{s=1}^{p} (A(x_s)+\theta(s,s)B_0(x_s))+\sum_{1\leq s<s'\leq p} \theta(s,s')B(x_s,x_{s'})\right\},\end{equation}
for known functions $A,\;B_0:\;\Xset\to \rset$ and a symmetric function
$B:\;\Xset\times\Xset\to \rset$, where $\Xset$ is a compact (generally finite) set. The real-valued symmetric matrix
$\theta=\{\theta(s,s'),\;1\leq s, s'\leq p\}$ is the network structure and is the parameter of interest.  The term $Z_\theta$ is
a normalizing constant.  This type of statistical models was pioneered by J. Besag (\cite{besag74}) under the name of \underline{auto-model} and we adopt the same name here, although Besag's auto-models corresponds to setting $B(x,y)=xy$ above.  The nice feature of model (\ref{model}) is that for any $1\leq s\leq p$, the conditional density of $X_s$ given $\{X_j,j\neq s\}=x\in\Xset^{p-1}$ is
\begin{equation}\label{condmodel}
f^{(s)}_\theta(u\vert x)=\frac{1}{Z^{(s)}_\theta}\exp\left\{A(u) + \theta(s,s)B_0(u) +\sum_{j\neq i}\theta(s,j)B(u,x_j)\right\},\end{equation}
for a normalizing constant $Z^{(s)}_\theta=Z^{(s)}_\theta(x)$. Therefore, $\theta(s,j)=0$ implies that $X_s$ and $X_j$ are conditionally independent given the other variables $X_k$, $k\notin \{s,j\}$. Thus estimating $\theta$ provides us with the dependence structure and the magnitude of the dependence between these variables. 

This paper focuses on the situation where the outcomes $X_j^{(i)}$ are either categorical ($\Xset$ is a finite set) or continuous bounded ($\Xset\subset\rset^{m_\Xset}$ is compact).  Based on $(X^{(1)},\ldots,X^{(n)})$, the true network structure denoted $\theta_\star=\{\theta_\star(s,s'),\;1\leq s, s'\leq p\}$ can be consistently estimated using a number of methods, even when the number of entries of $\theta_\star$ is much large than $n$ (\cite{hoefling09,ravikumaretal10,guoetal10}). For computational tractability, a pseudo-likelihood approach is often preferred, even though it  incurs a certain lost of efficiency. 
In the case of the auto-logistic model (where $\Xset=\{0,1\}$, $A_0(u)=0$, $B_0(u)=u$, $B(u,v)=uv$), \cite{guoetal10} shows  that the $\ell^1$-penalized pseudo-likelihood estimator of $\theta_\star$  is consistent with $\ell^2$ rate of convergence bounded from above by $\alpha^{-1}\sqrt{a\log p/n}$, where $a$ is the number of non-zero elements of $\theta_\star$ and $\alpha$ is the smallest eigenvalue of the information matrix. \cite{ravikumaretal10} obtained similar results for a one-neighborhood-at-the-time $\ell^1$-penalized pseudo-likelihood estimator. \cite{xueetal10} also derived some properties of the oracle estimator with the SCAD penalty.

In many situations where network estimation is needed, the network data is only partially observed because certain nodes are missing from the sample. For example, in social network analysis,  some close friends or siblings might not be part of the survey. As another example, in protein-protein networks, the analysis  is often restricted to the specific subgroup of proteins that is believed to carry a role in a given biological function. So doing, some important but not yet identified proteins might be omitted from the analysis. This paper consider the problem of network estimation from partially observed MRF data. The issue cannot be completely addressed by simply ignoring the missing nodes and assuming that the observed data follows a MRF.  This is because,  unlike Gaussian distributions,  Markov Random Field distributions are not closed under marginalization. For example, if there exist $r$  additional  nodes denoted $p+1,\ldots,p+r$ such that the joint distribution of $(X_1,\ldots,X_{p},X_{p+1},\ldots,X_{p+r})$ is an auto-model with network structure $\{\theta(s,s') ,\;1\leq s,s'\leq p+r\}$, then the joint (marginal) distribution of $(X_1,\ldots,X_{p})$ is \underline{not} of the form  (\ref{model}) in general. To take a specific example, if $r=1$ and $A=B_0\equiv 0$ and $B(x,y)=B(x)B(y)$, then the joint (marginal) distribution of $(X_1,\ldots,X_{p})$ is the mixture distribution
\[f_\theta(x_1,\ldots,x_{p})=Z^{-1}_\theta\sum_{i\in \Xset}\exp\left\{\sum_{s=1}^{p}\theta_i(s)B(x_s) + \sum_{1\leq s<s'\leq p} \theta(s,s')B(x_s)B(x_{s'})\right\},\]
where $\theta_i(s)=B(i)\theta(s,p+1)$. Furthermore, the conditional distributions are altered. Indeed, and keeping with the assumption $r=1$, if $|\theta(s,p+1)|>0$, then the conditional density of $X_s$ given $\{X_\ell,\;\ell\neq s,\,1\leq\ell\leq p\}$ depends not only $X_\ell$ for all $\ell$ such that $|\theta(s,\ell)|>0$, but also on $X_k$ for all $k$ such that $|\theta(k,p+1)|>0$. However, if $\theta(s,p+1)=0$, the conditional density of $X_s$  given $\{X_\ell,\;\ell\neq s,\,1\leq\ell\leq p\}$  remains (\ref{condmodel}). This suggests that if we ignore the missing nodes and fit the misspecified model (\ref{model}) to the observed data, the resulting estimator will be well-behaved to the extent that the missing data problem is limited. That is, to the extent that $\sum_{s=1}^p|\theta_\star(s,p+1)|$ is small in the case $r=1$  considered above.

The goal of the paper is to formalize this idea. In order to do so,  we consider an infinite-volume Markov random field model, where only part of the field is observed, and we fit the misspecified model (\ref{model}) using penalized pseudo-likelihood approach. We derive a general consistency result and show that under certain conditions, the estimators converges at the rate of $(\sqrt{a_n\log p_n/n} + \tau_n b_n)/\alpha_n$, where $p_n$ is the number of observed nodes, $a_n$ is the number of non-zero entries of the true network, $\alpha_n$ is the smallest eigenvalue of the Fisher information matrix, and where the term $\tau_nb_n$ quantifies the effect of the missing nodes (see Theorem \ref{thm3} for a more rigorous statement). We conclude that the estimator $\hat\theta_n$ is robust to a small to moderate amount of missing data. We report some simulation results that are consistent with these findings. In practical situations where MRF are used, it is often unclear whether one is dealing with a partially observed field with important missing nodes. The above discussion thus stresses the need for methods of detecting the existence of missing nodes in Markov random field data. We leave this problem for future research, as it requires a better understanding of the asymptotic behavior  of $\hat\theta_n$. 

The paper is organized as follows. The infinite-volume Markov random field setting and the estimators are presented in Section \ref{setting}. The paper presents two main results: Theorem \ref{thm2} (and Corollary \ref{coro1}) on the consistency of the estimator, and Theorem \ref{thm3} (and Corollary \ref{corothm3}) on its rate of convergence. These results are presented in Section \ref{sec:asympdist}. The simulation example is presented in Section \ref{sec:ex}. Section \ref{sec:proofs} develops the technical proofs.

\subsection{The setting}\label{setting}
Let $(\Xset,\e,\rho)$ be a measure space. We assume that $\Xset$ is a compact subset of $\rset^{m_\Xset}$, $\e$ its Borel sigma-algebra, and $\rho$ a finite measure.  The compactness of $\Xset$ is wrt the usual Euclidean metric.  $\Xset$ is the sample space of the observations $X_i$. The main case of interest is the case where $\Xset$ is finite.   Let $\S$ be a countably infinite set (typically, $\S$ is a subset of the Euclidean space $\rset^{m_\S}$  for some finite integer $m_\S\geq 1$).   The set $\S$ represents the nodes of the network. We assume that $\S$ is equipped with a linear ordering $\succeq$ (for example, the lexicographical ordering of $\rset^{m_\S}$). We introduce $\underline{\S}^2\eqdef \{(s,\ell)\in\S\times\S:\; \ell\succeq s\}$,  the set of all ordered pairs of $\S$. More generally, if $\Lambda$ is a subset of $\S$, we denote by
$\underline{\Lambda}^2$, the set of all ordered pairs $(u,v)\in\Lambda\times\Lambda$, with $v\succeq u$.   

Let
$A,\,B_0:\Xset\to\rset,\;B:\;\Xset\times\Xset\to\rset$ be \underline{known} measurable functions such that
$B(x,y)=B(y,x)$ (symmetry). We also assume that the diagonal of $B$ is $B_0$: $B(x,x)=B_0(x)$ for all $x\in\Xset$. We assume \underline{throughout the paper}  that
\begin{equation}\label{boundBC}
\|A\|_\infty<\infty,\;\;\;\|B_0\|_\infty<\infty,\;\;\;\mbox{ and }\;\;\;\|B\|_\infty<\infty.\end{equation}
In the above, $\|f\|_\infty$ is the supremum norm.

An infinite matrix is a map from $\S\times\S$ to $\rset$. For an infinite matrix $\theta:\;\S\times\S\to\rset$  and $s\in\S$, the $\theta$-neighborhood of $s$ is the set
\[\partial_\theta s\eqdef \{\ell\in\S:\; \ell\neq s \mbox{ and } |\theta(s,\ell)|>0\},\]
and the $\theta$-degree of node $s$ is the quantity (possibly infinite)
\[\textsf{deg}(s,\theta)\eqdef \sum_{\ell\in\S\setminus\{s\}}|\theta(s,\ell)|=\sum_{\ell\in\partial_\theta s}|\theta(s,\ell)|.\]
We denote 
$\M$ the space of all infinite symmetric matrices $\theta$ such that  $\deg(s,\theta)<\infty$ for all $s\in\S$.  For $q\in [1,\infty)$, we denote by $\M_{q}$ the Banach space of all infinite symmetric matrices $\theta\in\M$ such that
\[\|\theta\|_{q}\eqdef \left\{\sum_{(s,\ell)\in\underline{\S}^2}|\theta(s,\ell)|^q\right\}^{1/q}<\infty.\]

Let $(\Omega,\F)=(\Xset^\S,\e^\S)$ be the product space equipped with the product topology and its Borel sigma-algebra. For $\theta\in\M$, let $\mu_\theta$ be the probability measure on $(\Omega,\F)$ such that if $\{X_s,\;s\in\S\}$ is a stochastic process with distribution $\mu_\theta$, the conditional distribution of $X_s$ given the sigma-algebra generated by $\{X_\ell,\;\ell\neq s\}$ exists and has density (wrt $\rho$) $f_\theta^{(s)}(\cdot\vert x)$, where for $u\in\Xset$, $x\in\Xset^{\S\setminus\{s\}}$,
\begin{equation}\label{fullcond}
f^{(s)}_\theta(u\vert x)=\frac{1}{Z^{(s)}_\theta}\exp\left\{A(u)+\theta(s,s)B_0(u)+\sum_{\ell\in\S\setminus\{s\}}\theta(s,\ell)B(u,x_\ell)\right\},\end{equation}
for a normalizing constant $Z^{(s)}_\theta$. Notice that $f^{(s)}_\theta(u\vert x)$ actually depends only on $x_{\partial_\theta s}\eqdef\{x_\ell:\; \ell\in\partial_\theta s\}$. Under
(\ref{boundBC}) and for $\theta\in\M$, such distribution $\mu_\theta$
exists (but might not be unique in general). We refer the reader to Appendix 1 for a precise definition and
existence of $\mu_\theta$. A random process $\{X_s,\;s\in\S\}$ with
distribution $\mu_\theta$ is called an infinite-volume auto-model random
  field. We denote by $\PE_\theta$ the expectation operator with
respect to $\mu_\theta$ on $(\Omega,\F)$. When $\theta$ is the true network structure $\theta_\star$ (introduced below), we simply write $\PE_\star$ instead of $\PE_{\theta_\star}$. For $\Lambda\subseteq \S$, we denote $X_\Lambda$ the stochastic process $\{X_s,\;s\in\Lambda\}$. From $f_\theta^{(s)}$, and for a measurable function $H:\Xset\times \Xset^{\S\setminus\{s\}}\to\rset$, we can obtain the conditional expectation $\PE_\theta\left(H(X_s,X_{\S\setminus\{s\}})\vert X_{\S\setminus\{s\}}\right)$ as $\int_{\Xset}H(u,X_{\S\setminus\{s\}})f^{(s)}_\theta(u\vert X_{\S\setminus\{s\}})du$, provided the integral is well defined. And we can define similarly the conditional variance $\textsf{Var}_\theta\left(H(X_s,X_{\S\setminus\{s\}})\vert X_{\S\setminus\{s\}}\right)$.

For $\theta_\star\in\M$, let $\{X^{(i)},\;i\geq 1\}$ be a sequence of i.i.d.  infinite-volume random fields with distribution $\mu_{\theta_\star}$ defined on some probability space with probability measure $\cPP_{\star}$ and expectation operator $\cPE_\star$. Let $\{D_n,\;n\geq 1\}$ be a sequence of increasing finite subsets of $\S$ such that $D_n\uparrow\S$. For a finite set $A$, $|A|$ denotes its cardinality and we set $p_n=|D_n|$.  For $n\geq 1$, let $d_n=p_n(p_n+1)/2$ and denote $\M^{(n)}$ the space of all symmetric finite matrices $\{\theta(s,\ell),\;s,\ell\in D_n\}$, that we identify with $\rset^{d_n}$.

We assume that for some $n\geq 1$, we observe partially each of the random field $X^{(i)}$ ($1\leq i\leq n$) over the domain $D_n$ giving rise to observations $X_{D_n}^{(i)}=\{X^{(i)}_s,\;s\in D_n\}$. The remaining points $\S\setminus D_n$ are not  known and the associated random variables $X_{\S\setminus D_n}$ are not observed.  We are interested in estimating the infinite matrix $\theta_\star$.  For $s\in\S$, we define $\partial s=\partial_{\theta_\star} s$ and called it the (true) neighborhood of $s$. We also define $\partial_n s\eqdef D_n\setminus\{s\}$. Since the neighborhood system $\{\partial s,\;s\in\S\}$ is not known,  we introduce the \underline{approximate} full conditional distributions
\begin{equation}\label{approxfullcond}
f^{(s)}_\theta(u\vert x_{\partial_n s})\eqdef \frac{1}{Z_{n,\theta}^{(s)}}\exp\left(A(u) +\theta(s,s)B_0(u)+\sum_{\ell\in \partial_n s}\theta(s,\ell)B(u,x_\ell)\right),\end{equation}
for some normalizing constant $Z_{n,\theta}^{(s)}$. For $\lambda\geq 0$, let $q_\lambda:\; [0,\infty)\to [0,\infty)$ a penalty function. We then define the functions
\[
\bar \ell_n(\theta)\eqdef \sum_{i=1}^n\sum_{s\in D_n} \log f^{(s)}_\theta(X^{(i)}_s\vert X^{(i)}_{\partial_n s}),\;\;\mbox{ and } \;\; Q_n(\theta)= \bar \ell_n(\theta)-\sum_{(s,\ell)\in \underline{D}_n^2}q_{\lambda_n}(|\theta(s,\ell)|),\;\;\theta\in\M^{(n)},\]
for some parameter $\lambda_n>0$.   We are mainly interested in convex penalty functions, particularly the $\ell^1$ penalty for which $q_\lambda(x)=\lambda x$. But we develop much of the results under the general condition A\ref{A0} below that applies in principle to non-convex penalties such as the SCAD penalty of \cite{fanetli01}. 

\vspace{0.2cm}
\debutA
\item \label{A0}
For any $\lambda\geq 0$, $q_\lambda(0)=0$, $q_\lambda$ is right-continuous at $0$ and  differentiable on $(0,\infty)$ and 
\begin{equation}\label{condq}
\sup_{\lambda>0}\sup_{x>0}|q'_\lambda(x)|/\lambda<\infty. \end{equation}
\finA
\medskip

 Finally, we define 
\[\textsf{Argmax}\, Q_n\eqdef \{\theta\in\M^{(n)}:\; Q_n(\theta)=\sup_{\vartheta\in\M^{(n)}} Q_n(\vartheta)\},\]
and we call any element $\hat\theta_n$ of $\textsf{Argmax}\, Q_n$ a maximizer of $Q_n$, that is a penalized pseudo-likelihood estimator of $\theta_\star$.

\begin{rem}
We want to stress the fact that the sets $\S$ and $D_n$ are purely conceptual and need not be known. This is because we have replaced the full conditional density (\ref{fullcond}) by the approximation (\ref{approxfullcond}) in which the neighborhood of $s$ is $\partial_n s=D_n\setminus\{s\}$, and without any loss of generality we can replace $D_n$ by $\{1,\ldots,p_n\}$. As a result, the computation of $\hat\theta_n$ does not make use of $\S$ and $D_n$. For instance, with the $\ell^1$ penalty, one obtains  the same $\ell^1$-penalized pseudo-likelihood estimator as in \cite{hoefling09,guoetal10}. 
\end{rem}

It is useful to have some simple conditions under which $\textsf{Argmax}\, Q_n$ is not empty. 

\begin{prop}Fix $n\geq 1$.  Suppose that for any $s\in\S$, there exists a finite constant $c(s)$ such that for all $\theta\in\M^{(n)}$, all $u\in \Xset$ and for all $x_{\partial_n s}\in \Xset^{\partial_n s}$,
\[f^{(s)}_\theta(u\vert x_{\partial_n s})\leq c(s).\]
Suppose also that for any $\alpha,\lambda\geq 0$ the set $\{x\geq 0:\; q_\lambda(x)\leq \alpha\}$ is bounded. Then  $\textsf{Argmax}\; Q_n$ is non-empty.
\end{prop}

\begin{rem}
The result is not always useful. It applies to the $\ell^1$ penalty but not to the SCAD penalty. If $\Xset$ is finite as in all the examples below, then $f^{(s)}_\theta(\cdot\vert x_{\partial_n s})$ is a finite probability mass function. Therefore the assumption of the proposition holds with $c(s)=1$.
\end{rem}

\begin{proof}
 Fix a sample path $\omega\in\Pi$. Then $Q_n$ is a continuous $\rset$-valued function on $\M^{(n)}$. Denote $\textsf{0}$ the null element of $\M^{(n)}$, and $r=Q_n(\textsf{0})$. Then $\L_r\eqdef\{\theta\in\M^{(n)}:\; Q_n(\theta)\geq r\}$  is nonempty and closed by continuity of $Q_n$. Under the assumption of the proposition, if $\theta\in\L_r$, then for any $(s,\ell)\in\underline{D}_n^2$:
 \[q_{\lambda_n}(|\theta(s,\ell)|)\leq \sum_{(s,\ell)\in \underline{D}_n^2}q_{\lambda_n}(|\theta(s,\ell)|)\leq n\sum_{s\in D_n}\log c(s)-r.\]
  Thus $\L_r$ is a compact subset of $\M^{(n)}$ and $Q_n$ attains it maximum at $\hat\theta_n\in\L_r$. 
\end{proof}

\subsection{Consistency and rate of convergence}\label{sec:asympdist} 
Let $\M_1$ be the separable Banach space of all  $\theta\in\M$ such that $\|\theta\|_1\eqdef \sum_{(u,v)\in\underline{\S}^2}|\theta(u,v)|<\infty$. We  investigate the consistency of $\hat\theta_n$ as a random element of $\M_1$ under the following sparsity assumption.

\vspace{0.2cm}
\debutA
\item \label{A1} $\theta_\star\in \M$  and for any $s\in\S$, the $\theta_\star$-neighborhood of $s$ (that is, the set $\partial_{\theta_\star} s=\{\ell\in\S\setminus\{s\}:\;\theta_\star(s,\ell)\neq 0\}$) is a finite set.
\finA

\vspace{0.2cm}
A\ref{A1} guarantees that for $\theta\in\M_1$, $\theta+\theta_\star\in\M$, so that the full conditional densities $f^{(s)}_{\theta+\theta_\star}(u\vert x_{\S\setminus \{s\}})$ are well defined. For two matrices $\theta,\theta'\in\M$, we write $\theta\cdot\theta'$ to denote the component-wise product. And if $\theta\in\M$, and $n\geq 1$, $\theta^{(n)}$ denotes the element of $\M^{(n)}$ such that $\theta^{(n)}(u,v)=\theta(u,v)$ if $(u,v)\in D_n\times D_n$ (and $\theta^{(n)}(u,v)=0$ otherwise). We introduce
\begin{multline}\label{fUn}
U_n(\theta)\eqdef n^{-1}\sum_{i=1}^n\sum_{s\in D_n} \left(\log f^{(s)}_{\theta_\star}(X^{(i)}_s\vert X^{(i)}_{\partial_n s})-\log f^{(s)}_{\theta_\star+\theta}(X^{(i)}_s\vert X^{(i)}_{\partial_n s})\right)\\
+n^{-1}\sum_{(s,\ell)\in\underline{D}_n^2}\left(q_{\lambda_n}(|\theta_\star(s,\ell)+\theta(s,\ell)|)-q_{\lambda_n}(|\theta_\star(s,\ell)|)\right),\;\;\theta\in\M_1.\end{multline}
$U_n(\theta)$ is no other than $n^{-1}\left(Q_n(\theta_\star)-Q_n(\theta_\star+\theta)\right)$ and is minimized at $\hat\theta_n-\theta_\star^{(n)}$. We also introduce the conditional Kulback-Leibler divergence function
\begin{equation}\label{kstheta}
k^{(s)}(\theta_\star,\theta)\eqdef\PE_{\theta_\star}\left(\int-\log\left(\frac{f_{\theta_\star+\theta}^{(s)}(u\vert X_{\S\setminus\{s\}})}{f_{\theta_\star}^{(s)}(u\vert X_{\S\setminus\{s\}})}\right)f_{\theta_\star}^{(s)}(u\vert X_{\S\setminus\{s\}})du\right),\;\;\theta\in\M_1.\end{equation}
By the concavity of the logarithm function, $k^{(s)}(\theta_\star,\theta)\geq 0$. Finally, we define
\begin{equation}\label{ktheta}
k_n(\theta_\star,\theta)\eqdef \sum_{s\in D_n}k^{(s)}(\theta_\star,\theta),\;\mbox{ and }\;\;k(\theta_\star,\theta)\eqdef \sum_{s\in\S}k^{(s)}(\theta_\star,\theta).\end{equation}
Clearly, $k_n(\theta_\star,\theta)$ is nondecreasing in $n$ and converges to $k(\theta_\star,\theta)$. For any $\theta\in\M_1$ and $u\in\Xset$, we use (\ref{Id:PS}) and (\ref{boundBC}) to verify that
\begin{multline*}
\left|\log f^{(s)}_{\theta_\star+\theta}(u\vert X_{\S\setminus\{s\}})-\log f^{(s)}_{\theta_\star}(u\vert X_{\S\setminus\{s\}})\right|\leq \\
\left|\theta(s,s)B_0(u) + \sum_{\ell\in\S\setminus\{s\}}\theta(s,\ell)B(u,X_\ell)\right|
 +\left|\log Z^{(s)}_{n,\theta+\theta_\star}-\log Z^{(s)}_{n,\theta_\star}\right| \\\leq C\left(|\theta(s,s)|+\sum_{\ell\in\partial_\theta s}|\theta(s,\ell)|\right),\end{multline*}
for some finite constant $C$. This  implies that $\sum_{s\in D_n} k^{(s)}(\theta_\star,\theta)\leq C\|\theta\|_1<\infty$ and proves that $k(\theta_\star,\theta)$ is finite.  Also notice that $\textsf{Argmin}\; k(\theta_\star,\cdot)$ is nonempty and contains the null matrix $\textsf{0}$.

We study the consistency of $\hat\theta_n$ using epi-convergence methods as in \cite{hess96}. We review some definitions. For more on epi-convergence, we refer to \cite{dalmaso93}. Let $(\Vset,\dist)$ be a metric space and $\{f,f_n,\;n\geq 1\}$ be a sequence of functions defined on $\Vset$, and taking values in $\rset\cup\{-\infty,+\infty\}$. The epi-limit inferior of $\{f_n,\;n\geq 1\}$ is the function
\[\textsf{li}_e f_n(x)=\sup_{k\geq 1}\;\liminf_{n\to\infty}\inf _{v\in B(x,k^{-1})} f_n(v),\]
where $B(x,\epsilon)$ denotes the open ball of $\Vset$ with center $x$ and radius $\epsilon$. We define similarly the epi-limit superior of $\{f_n,\;n\geq 1\}$ as
\[\textsf{ls}_e f_n(x)=\sup_{k\geq 1}\;\limsup_{n\to\infty}\inf _{v\in B(x,k^{-1})} f_n(v).\]
We say that $f_n$ epi-converges (or $\Gamma$-converges)  to $f$ if $\textsf{ls}_e f_n(x)\leq f(x)\leq \textsf{li}_e f_n(x)$ for all $x\in\Vset$. 

\begin{theo}\label{thm2}
Assume A\ref{A0}-\ref{A1}, (\ref{boundBC}) and suppose that $n^{-1}\lambda_n=o(1)$, as $n\to\infty$. Then almost surely, $U_n$ epi-converges to $k(\theta_\star,\cdot)$ in $(\M_1,\|\cdot\|_1)$.
\end{theo}
\begin{proof}
See Section \ref{proofthm2}.
\end{proof}

Epi-convergence is a very useful tool in the study of minimizers. The key result in that respect is as follows (using the notations of the above paragraph). If $f_n$ epi-converges to $f$ and $\{x_n,\;n\geq 1\}$ is such that $x_n\in \textsf{Argmin}\,f_n$, then if $x_n\to \bar x$ (in the metric space $(\Vset,\dist)$), $\bar x\in \textsf{Argmin}\,f$.  In order to make use of this result in our case, we need to impose additional conditions that ensure that $\hat\theta_n$ converges and that the limiting function $k(\theta_\star,\cdot)$ admits a unique minimum.  

For $s\in\S$ and $\theta\in\M$, define  the infinite matrix $\rho_\theta^{(s)}\eqdef\{\rho_\theta^{(s)}(\ell,\ell'),\;\ell,\ell'\in \S\}$, where 
\[\rho_\theta^{(s)}(\ell,\ell')\eqdef \PE_\star\left[\textsf{Cov}_\theta\left( B(X_s,X_\ell),B(X_s,X_{\ell'})\vert X_{\partial_\theta s}\right)\right],\;\;\ell,\ell'\in\S.\]

%
%

\begin{coro}\label{coro1}
Suppose that the assumptions of Theorem \ref{thm2} hold and also that for any $\theta\in\M$, any $s\in\S$, $\rho_\theta^{(s)}$ is a positive definite matrix. Let $\{\hat\theta_n,\;n\geq 1\}$ be a Borel measurable sequence of $\M_1$ such that $\hat\theta_n\in \textsf{Argmax}\, Q_n$. If  $\{(\hat\theta_n-\theta_\star^{(n)}),\;n\geq 1\}$ is uniformly tight, as a random sequence of $\M_1$, then $\|\hat\theta_n-\theta_\star^{(n)}\|_1$ converges in probability to zero.
\end{coro}
\begin{proof}
See Section \ref{proofcoro1}.
\end{proof}

The tightness condition is needed but in general is difficult to check. Intuitively,  the tightness of $\{(\hat\theta_n-\theta_\star^{(n)}),\;n\geq 1\}$ implies that the overall dependence between the missing nodes and the observed nodes is limited. We will not attempt to make this statement precise.  We will rather study more precisely the connection between the missing nodes and the rate of convergence of $\hat\theta_n$.  We assume that the following holds (see Section \ref{assumB2} for a discussion).

\vspace{0.2cm}

\debutA
\item \label{B2} Assume that there exist $\alpha_n,\alpha_n'>0$ such that for  all $\theta,\theta'\in\M^{(n)}$,
\begin{multline*}
\sum_{s\in D_n}\PE_{\theta_\star}\left[\textsf{Var}_{\theta'}\left(\sum_{\ell\in D_n}\theta(s,\ell)B(X_s,X_\ell)\vert X_{\partial_n s}\right)\right] \geq \alpha_n\|\theta\|_2^2,\\
 \mbox{ and }\;\;\; \PE_\star^{1/2}\left[\left(\sum_{s\in D_n}\log\left(\frac{f_{\theta^{(n)}_\star+\theta}^{(s)}\left(X_s^{(1)}\vert X_{\partial_n s}^{(1)}\right)}{f_{\theta^{(n)}_\star}^{(s)}\left(X_s^{(1)}\vert X_{\partial_n s}^{(1)}\right)}\right)\right)^2\right]\leq \alpha_n'\|\theta\|_2,\end{multline*}
for all $n$ large enough.
\finA

\vspace{0.2cm}

Let $\{a_n,\;n\geq 1\}$ be a sequence of positive numbers. We define $\M^{(n)}(a_n)$ the set of all finite $p_n\times p_n$ symmetric matrix $\theta$  such that
\[\left|\left\{(s,\ell)\in\underline{D}_n^2:\; |\theta(s,\ell)|>0\right\}\right|\leq a_n.\]
In other words,  $\M^{(n)}(a_n)$ is the set of elements of $\M^{(n)}$ with sparsity $a_n$. We introduce $\Delta_n^{(c)}\eqdef \{s\in D_n:\; \partial s\setminus D_n \neq \emptyset\}$, the set of observed nodes that admit neighbors outside $D_n$. We can think of $\Delta_n^{(c)}$ as the boundary of $D_n$. Let $\{\tau_n,\;n\geq 1\}$ be another sequence of positive numbers.  We define $\M^{(n)}(a_n,\tau_n)$ as the set of all $\theta\in \M^{(n)}(a_n)$  such that
\[\left\{\sum_{s\in \Delta_n^{(c)}}\left(\sum_{\ell\in D_n} |\theta(s,\ell)|\right)^2\right\}^{1/2}\leq \tau_n \left\{\sum_{(s,\ell)\in \underline{D}_n^2}|\theta(s,\ell)|^2\right\}^{1/2}.\]

We relate the behavior of the estimator $\hat\theta_n$, to the class of functions $\F_{n,\delta}\eqdef \{m_{n,\theta},\;\theta\in B_{n,\delta}\}$, where $B_{n,\delta}\eqdef \{\theta\in \M^{(n)}(a_n,\tau_n):\;\|\theta\|_2\leq \delta\}$, $\delta>0$, and
\begin{equation}\label{defmn}
m_{n,\theta}(x)=\sum_{s\in D_n}\log \left(\frac{f^{(s)}_{\theta_\star}(x_s\vert x_{\partial_n s})}{f^{(s)}_{\theta_\star+\theta}(x_s\vert x_{\partial_n s})}\right),\;\; x\in\Xset^\infty.\end{equation}
It is clear that the size of the family $\F_{n,\delta}$ depends on the size of $B_{n,\delta}$. By the sparsity of $\M^{(n)}(a_n,\tau_n)$, and for $a_n\leq d_n/ 2$ (we recall that $d_n=p_n(p_n+1)/2$, where $p_n=|D_n|$), we have
\begin{equation}\label{coveringnumber}
N(\epsilon,B_{n,\delta},\|\cdot\|_{2})\leq \left(\frac{c\delta d_n}{\epsilon a_n}\right)^{a_n},\end{equation}
for some universal constant $c$, where  $N(\epsilon,B_{n,\delta},\|\cdot\|_{2})$ denotes the $\epsilon$-covering number of the set $B_{n,\delta}$ with respect to the $\ell^2$-norm on $\M_n(a_n,\tau_n)$. To see this, notice that the $\epsilon$-covering number of the $\ell^2$-ball of  $\rset^{a_n}$ with radius $\delta$ is bounded from above by $(3\delta/\epsilon)^{a_n}$. 
 For $\theta\in\M^{(n)}(a_n,\tau_n)$, since the number of non-zeros entries of $\theta$ is bounded from above by $a_n$, there are  at most ${d_n \choose a_n}$  ways of forming $\theta$ from a sequence of $a_n$ non-zeros elements of $\rset^{a_n}$. Thus $N(\epsilon,B_{n,\delta},\|\cdot\|_{2})\leq{d_n\choose a_n}(3\delta/\epsilon)^{a_n}$. By Stirling's formula, 
 \begin{multline*}
 {d_n\choose a_n}\leq \sqrt{\frac{2c}{a_n}}\exp\left(-a_n\log(a_n/d_n) -(d_n-a_n)\log(1-a_n/d_n)\right)\\
 \leq \exp\left(a_n\left(\log(d_n)-\log(a_n) +c\right)\right),\end{multline*}
for some finite constant $c$,  which leads to (\ref{coveringnumber}). See also \cite{roman09}. Finally we introduce
\[b_n\eqdef \left\{\sum_{s\in D_n}\left(\sum_{\ell\in \partial s\setminus D_n}|\theta_\star(s,\ell)|\right)^2\right\}^{1/2},\]
 which measure the strength of the dependence between the missing nodes and the observed nodes. Our main result is as follows.
  
\begin{theo}\label{thm3}
Assume (\ref{boundBC}), A\ref{B2}. Suppose that as $n\to\infty$, $a_n\sqrt{\log p_n}=O(\alpha_n'n^{1/2})$, and $\lambda_n=O(\sqrt{n\log p_n})$, and also that 
$(\hat\theta_n-\theta_\star^{(n)})\in \M_n(a_n,\tau_n)$ for all $n$ large enough. Then $r_n\|\hat\theta_n-\theta_\star^{(n)}\|_2=O_p(1)$ as $n\to\infty$, where $r_n=\alpha_n\sqrt{n}/\left(\sqrt{a_n\log p_n}+ \sqrt{n}b_n\tau_n\right)$. 
\end{theo}
\begin{proof}
See Section \ref{proofthm3}.

\end{proof}

The theorem implies that $\hat\theta_n$ is consistent in estimating $\theta_\star^{(n)}$ if $\tau_nb_n=o(\alpha_n)$. In the above result, we need to find $a_n$ and $\tau_n$ that guarantee that $(\hat\theta_n-\theta_\star^{(n)})\in \M_n(a_n,\tau_n)$ for all $n$ large enough.  Notice that for any $\theta\in\M^{(n)}$, we have trivially
\[\left\{\sum_{s\in \Delta_n^{(c)}}\left(\sum_{\ell\in D_n} |\theta(s,\ell)|\right)^2\right\}^{1/2}\leq 2\sup_{\{s\in \Delta^{(c)}_n\}}|\{\ell\in D_n:\; |\theta(s,\ell)|>0\}|^{1/2}\; \|\theta\|_2.\]
This means that any $\theta\in\M^{(n)}(a_n)$ also belongs to $\M^{(n)}(a_n,\tau_n)$, for $\tau_n=2\textsf{n}_n^{1/2}$, where $\textsf{n}_n\eqdef \sup_{\{\theta\in\M^{(n)}(a_n)\}} \sup_{s\in \Delta_n^{(n)}}|\{\ell\in D_n:\; |\theta(s,\ell)|>0\}|$. 
Therefore with this choice of $\tau_n$, we can replace $\M^{(n)}(a_n,\tau_n)$ by $\M^{(n)}(a_n)$ in the above theorem. This leads to the following reformulation.

\begin{coro}\label{corothm3}
Assume (\ref{boundBC}), A\ref{B2}. Suppose that as $n\to\infty$, $a_n\sqrt{\log p_n}=O(\alpha_n'n^{1/2})$, and $\lambda_n=O(\sqrt{n\log p_n})$, and also that 
$(\hat\theta_n-\theta_\star^{(n)})\in \M_n(a_n)$ for all $n$ large enough. Then $r_n\|\hat\theta_n-\theta_\star^{(n)}\|_2=O_p(1)$ as $n\to\infty$, where $r_n=\alpha_n\sqrt{n}/\left(\sqrt{a_n\log p_n}+ \sqrt{n}b_n\textsf{n}^{1/2}_n\right)$. 
\end{coro}

If the penalty function is $q_\lambda(x)=\lambda x$, the extensive recent literature on lasso points to the fact that the estimator $\hat\theta_n$ is sparse and recovers the sparsity structure of the true network $\theta_\star$ (\cite{barnejeeetal08,meinshausenetal09,guoetal10}). This suggest that $a_n$ can be taken proportional to the sparsity of $\theta_\star^{(n)}$. That is,
\[a_n\propto\left|\{(s,\ell)\in\underline{D}_n^2:\; |\theta_\star(s,\ell)|>0\}\right|.\]

Finally, we will point out that if $b_n=0$, then there is no missing data problem. In that case,  Theorem \ref{thm3} yields a similar rate of convergence as in \cite{guoetal10}.

\subsubsection{Comment on A\ref{B2}}\label{assumB2}
For $s\in D_n$, $\theta\in\M^{(n)}$, consider the following matrices  $\rho^{(s)}_{n,\theta}\eqdef \{\rho^{(s)}_{n,\theta}(\ell,\ell'),\;\ell,\ell'\in D_n\}$ and $\bar\rho_{n,\theta}\eqdef \{\bar\rho_{n,\theta}(s,\ell;s',\ell'),\;(s,\ell),(s',\ell')\in D_n\times D_n\}$, where
\begin{multline*}
\rho^{(s)}_{n,\theta}(\ell,\ell')\eqdef \PE_\star\left[\textsf{Cov}_\theta\left(B(X_s,X_\ell),B(X_s,X_{\ell'}\vert X_{\partial_n s}\right)\right],\;\mbox{ and }\,\bar\rho_{n,\theta}(s,\ell;s',\ell')\\
\eqdef \PE_\star\left[\left(B(X_s,X_\ell)-\PE_\theta\left(B(X_s,X_\ell)\vert X_{\partial_n s}\right)\right)\left(B(X_{s'},X_{\ell'})-\PE_\theta\left(B(X_{s'},X_{\ell'})\vert X_{\partial_n s'}\right)\right)\right].\end{multline*}
It can be easily seen that if the smallest eigenvalue of $\rho^{(s)}_{n,\theta}$ is bounded from below by $\alpha_n>0$, uniformly in $s$ and $\theta$, then the first part of A\ref{B2} holds. Similarly, if the largest eigenvalue of $\bar\rho^{(s)}_{n,\theta}$ is bounded from above by $\alpha_n'<\infty$, uniformly in $\theta$, then the second part of A\ref{B2} holds.

In many practical examples, $B(x,y)=B_0(x)B_0(y)$, where $B_0$ is a nonnegative and bounded function. In that case, one can often check A\ref{B2}. Indeed, the matrix $\rho^{(s)}_{n,\theta}$ becomes $\rho^{(s)}_{n,\theta}(\ell,\ell')=\PE_\star\left(\bar B_{s,\ell}\bar B_{s,\ell'}\textsf{Var}_\theta\left(B(X_s)\vert X_{\partial_n s}\right)\right)$, where $B_{s,\ell}=1$ is $\ell=s$ and $B_{s,\ell}=B_0(X_\ell)$ otherwise. If there exists $c_n>0$ such that $\textsf{Var}_\theta\left(B(X_s)\vert X_{\partial_n s}\right)\geq c_n$ for all $s\in D_n$ and all $\theta\in\M^{(n)}$, then the first part of A\ref{B2} holds with $\alpha_n=c_n\alpha_{n,0}$, where $\alpha_{n,0}>0$ is the smallest of the eigenvalues of the matrices $\PE_\star\left(\bar B_{s,\ell}\bar B_{s,\ell'}\right)$. Similarly,
\begin{multline*}
\bar\rho_{n,\theta}(s,\ell;s',\ell')= \PE_\star\left[\bar B_{s,\ell}\bar B_{s'\ell'}\left(B(X_s)-\PE_\theta\left(B(X_s)\vert X_{\partial_n s}\right)\right)\left(B(X_{s'})-\PE_\theta\left(B(X_{s'})\vert X_{\partial_n s'}\right)\right)\right]\\
\leq C\PE_\star\left(\bar B_{s,\ell}\bar B_{s'\ell'}\right).\end{multline*}
Then the second part of A\ref{B2} holds and we can take $\alpha_n'$ proportional to  the largest eigenvalue of $\PE_\star\left(\bar B_{s,\ell}\bar B_{s'\ell'}\right)$.

\begin{example}[The auto-binomial and auto-logistic models]

We consider here the particular case of the auto-binomial models which is an extension of the popular auto-logistic model. The auto-binomial model allows to model data where the available observation at each node can be seen as a number of successes over a given common number of trials. Fix $\kappa\geq 1$ the number of trials and set $\Xset=\{0,\ldots,\kappa\}$. The interaction functions of the auto-binomial model are given by $A(u)={\kappa \choose u}$, $B_0(u)=u$  and $B(u,v)=uv$. The particular case $\kappa=1$ corresponds to the auto-logistic model. The modeling assumption here is that for any $s\in\S$,
\[X_s\vert X_{\partial_\theta s}=x_{\partial_\theta s}\sim \mathcal{B}(\kappa,\alpha_\theta^{(s)}(x_{\partial_\theta s})),\;\mbox{ where } \;\log\left(\frac{\alpha_\theta^{(s)}(x_{\partial_\theta s})}{1-\alpha_\theta^{(s)}(x_{\partial_\theta s})}\right)=\theta(s,s) + \sum_{\ell\in\S\setminus\{s\}}\theta(s,\ell)x_\ell.\]
In the above display, $\mathcal{B}(n,p)$ denotes the binomial distribution with parameters $n,p$. Now for $\theta\in\M^{(n)}$, $\textsf{Var}_\theta\left(B(X_s)\vert X_{\partial_n s}\right)= \kappa\alpha_{n,\theta}^{(s)}\left(1-\alpha_{n,\theta}^{(s)}\right)$, where $\alpha_{n,\theta}^{(s)}$ is given by  $\alpha_{n,\theta}^{(s)}=\left(1+\exp\left(-\theta(s,s)-\sum_{j\neq s\,j=1}^{p_n}\theta(s,j)X_j\right)\right)^{-1}$. If we insist that $\sup_{s,\ell\in D_n}|\theta(s,\ell)|\leq K$, and that $\theta\in \M^{(n)}(a_n)$, then
\[\textsf{Var}_\theta\left(B(X_s)\vert X_{\partial_n s}\right)\geq 4^{-1}\kappa e^{-2K \textsf{N}_n^{1/2}},\;\;s\in D_n,\theta\in\M^{(n)}(a_n),\]
where $\textsf{N}_n\eqdef \sup_{\theta\in\M^{(n)}(a_n)}\sup_{s\in D_n}|\{\ell\in D_n:\; |\theta(s,\ell)|>0\}$, is the maximum degree in $\M^{(n)}(a_n)$. It follows that A\ref{B2} holds with $\alpha_n=\alpha_{n,0}4^{-1}\kappa e^{-2K \textsf{N}_n^{1/2}}$, where $\alpha_{n,0}>0$ is the smallest of the eigenvalues of the matrices $\PE_\star\left(\bar B_{s,\ell}\bar B_{s,\ell'}\right)$; and $\alpha_n'$ can be take as proportional to  the largest eigenvalue of $\PE_\star\left(\bar B_{s,\ell}\bar B_{s'\ell'}\right)$.

\end{example}

\subsection{Monte Carlo Evidence}\label{sec:ex}
We consider the auto-logistic model where $\Xset=\{0,1\}$, $A(x)=0$, $B_0(x)=x$, and $B(x,y)=xy$. We work with the $\ell^1$ penalty: $q_\lambda(x)=\lambda x$. With respect to the number of nodes, we consider two cases: $p=50 $ and $p=80$. For each setting, we consider different values of $n$ (the sample size) through the formula $n=a\log p/\beta^2$, where $a$ is the number of non-zero elements of the true network structure that we choose to be approximately $1.3*p$, and where $\beta$ is chosen in the range $[0.3,2.0]$ (for $p=50$), and $[0.6,2.0]$ (for $p=80$). 

We compare three settings. In Setting 1, there is no missing data, and the samples are generated exactly from (\ref{model}), for $\theta=\theta_\star$ (we set up $\theta_\star$ such that $\theta_\star(s,\ell)\geq 0$ and we use Propp-Wilson's perfect sampler). In Setting 2 and 3, we generate the sample $(X^{(i)}_1,\ldots,X^{(i)}_p,X^{(i)}_{p+1},\ldots,X^{(i)}_{p+r})$ from (\ref{model}), for $\theta=\theta_\star$, and we retain only $(X^{(i)}_1,\ldots,X^{(i)}_p)$, for $1\leq i\leq n$. Thus there are $r$ missing nodes. In Setting 2, we use $r=8$, whereas in Setting 3, we set $r=20$. Table 1 shows the corresponding values of $b_n$ in each setting.

\vspace{0.5cm}

\begin{table}[h]\label{table1}
\begin{center}
\small
\begin{tabular}{lccc}
\hline
&Setting 1, $r=0$ & Setting 2, $r=8$ & Setting 3, $r=20$\\
\hline
$p=50$ &0 & 1.8 & 4.41 \\
$p=80$ &0 & 1.8 & 3.6 \\
\hline
\end{tabular}
\caption{Values of $b_n$ in each setting of the simulation.}
\end{center}
\end{table}

\vspace{0.5cm}

Regardless of the data generation mechanism, we fit model (\ref{model})  by $\ell^1$ penalized pseudo-likelihood and compute the relative Mean Square Error $\PE_\star\left(\|\hat\theta-\theta_\star\|_2\right)/\|\theta_\star\|_2$, estimated from $K$ replications of the estimator ($K=50$). In Figure 1, we plot $\PE_\star\left(\|\hat\theta-\theta_\star\|_2\right)/\|\theta_\star\|_2$ as a function of $\beta$. As expected, the more missing data, the worst the estimator behaves. Notice that in Setting 2 (where $r=8$), the loss of accuracy of the estimator is worst for $p=50$ compared to $p=80$, although we have the same value $b_n=1.8$. This points to the fact that in the rate of convergence of $\hat\theta$, the factor $b_n$ is modulated by a factor related to size of the problem as predicted in Theorem \ref{thm3} (the term $\tau_n$).

\begin{center}
\rotatebox{0}{\scalebox{0.4}{\includegraphics{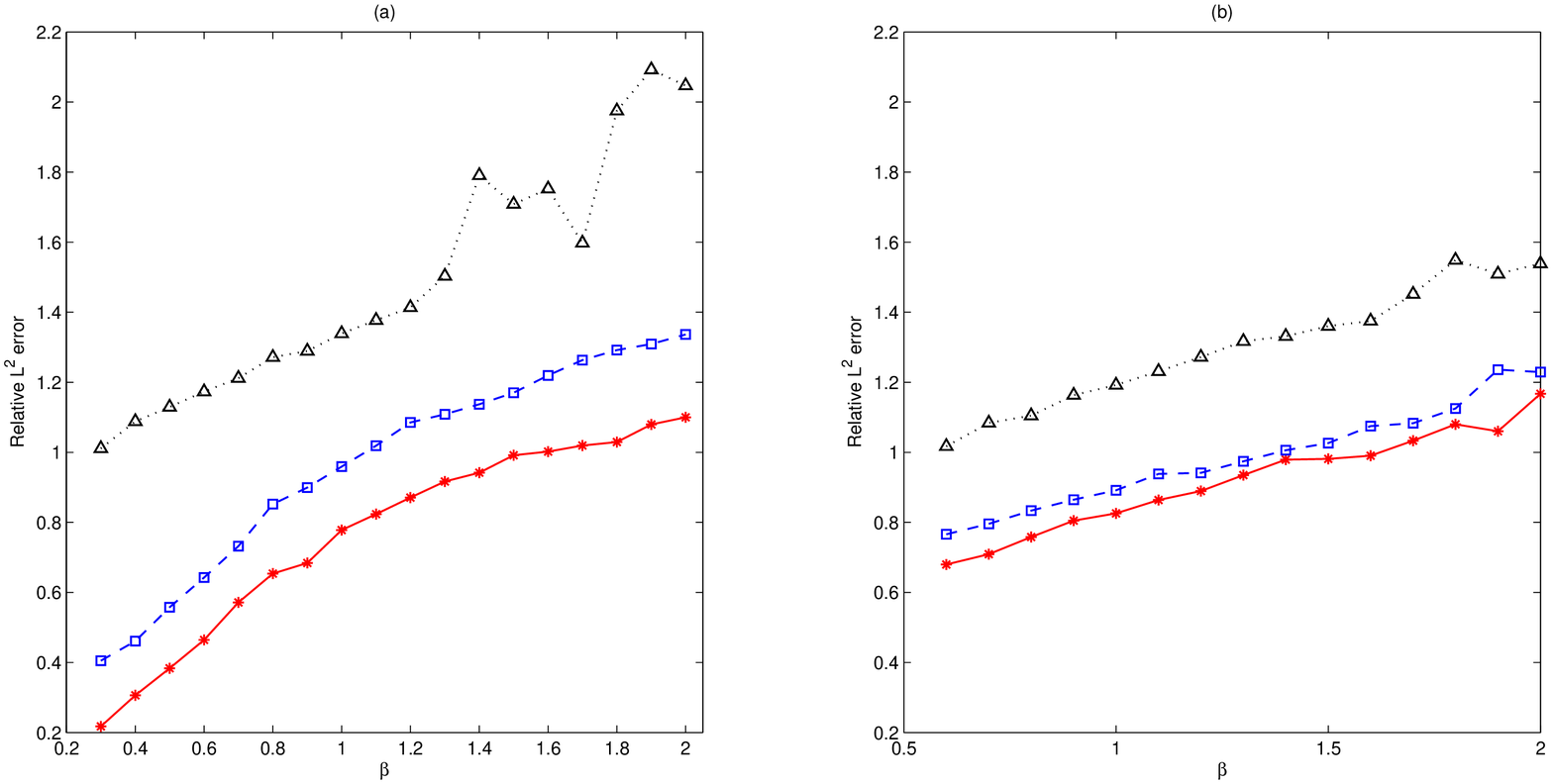}}}\\
\underline{Figure 1}:\footnotesize{ Relative $\textsf{MSE}$ versus $\beta$, where star-line is Setting 1, square-line is Setting 2, triangle-line is Setting 3. (a) $p=50$, (b) $p=80$.}
\end{center}

\section{Proofs}\label{sec:proofs}
\subsection{Some basic facts on infinite volume auto-models}\label{infiniteautomodels}
We recall from  \cite{georgii88} some basic facts on Gibbs distributions. Let $(\Xset,\e,\rho)$ and $\S$ as in Section \ref{setting}.
Let $(\Omega,\F)=(\Xset^\S,\e^\S)$ be the product space equipped with the product Borel sigma-algebra. We will need few more notations. We denote by $X_s$ the projection maps, that is, $X_s:\;(\Omega,\F)\to (\Xset,\e)$ such that $X_s(\omega)=\omega_s$. For $A\subseteq\Xset$ and $\Delta\subset\S$, we denote by $A^\Delta$ the product set $\{(\omega_s)_{s\in\Delta},\;\omega_s\in A\}$. We define $X_\Delta:\;(\Omega,\F)\to(\Xset^\Delta,\e^\Delta)$ as $X_\Delta(\omega)=\{\omega_s,\;s\in\Delta\}$ and we denote $\F_\Delta$ the sub $\sigma$-algebra of $\F$ generated by the map $X_U$, $U\subseteq\Delta$, $U$ finite. For two disjoint subsets $\Lambda$, $\Delta$ of $\S$, if $u=\{x_i,\;i\in\Delta\}$ and $v=\{x_i,\;i\in\Lambda\}$ we write $uv=\{x_i,\;i\in \Delta\cup\Lambda\}$ for the the concatenation of $u$, $v$.

For $\Delta\subset\S$ finite, we define the kernel $\rho_\Delta$  from $(\Omega,\F_{\S\setminus\Delta})$ to $(\Omega,\F)$ as follows
\[\rho_\Delta(\omega,A)\eqdef \left(\rho^\Delta\times \delta_{\omega_{\S\setminus\Delta}}\right)(A)
=\rho^\Delta\left(\{u\in\Xset^\Delta:\; uw_{\S\setminus\Delta}\in A\}\right),\;\;\;\omega\in\Omega,\;A\in\F.\]
In the above, $\delta_x$ is the Dirac mass at $x$ and $\rho^\Delta$ denotes the product measure $\bigotimes_{s\in\Delta}\rho$ on $(\Xset^\Delta,\e^\Delta)$. This kernel is best understood through its operation on bounded functions. If $f:\;\Omega\to\rset$ is a bounded measurable function and $\omega\in\Omega$, we have
\[\rho_\Delta f(\omega)\eqdef \int \rho_\Delta(w,dz)f(z)=\int_{\Xset^\Delta} f(u\omega_{\S\setminus\Delta})\rho^\Delta(du).\]

For an infinite matrix $\theta:\;\S\times\S\to\rset$ and $\Lambda$ a finite subset of $\S$, we define $\pi_{\theta,\Lambda}$ a probability kernel from $(\Omega,\F_{\S\setminus\Lambda})$ to $(\Omega,\F)$ by
\[\pi_{\theta,\Lambda}(\omega,dz)=\frac{1}{Z_\theta(\omega)}\exp\left\{H_{\theta,\Lambda}(z)\right\}\rho_\Lambda(\omega,dz),\]
where for $z\in\Omega$,
\[H_{\theta,\Lambda}(z)=\sum_{s\in\Lambda}\left(A(z_s) + \theta(s,s)B_0(z_s)+ \sum_{\ell\succeq s,\ell\neq s} \theta(s,\ell)B(z_s,z_\ell)\right).\]
The term $Z_\theta(\omega)\eqdef\rho_{\Lambda} H_{\theta,\Lambda}(\omega)$ is the normalizing constant. We write  $\pi_\theta=\{\pi_{\theta,\Lambda},\;\Lambda\subset\S,\;\Lambda \mbox{ finite}\}$ assuming that each kernel $\pi_{\theta,\Lambda}$ is well defined. If $\mu$ is a probability measure on $(\Omega,\F)$, $h:\;\Omega\to\rset$ a $\mu$-integrable function and $\mathcal{G}$ a sub-sigma-algebra of $\F$, we denote by $\mu(h\vert \mathcal{G})$ the conditional expectation of $h$ given $\mathcal{G}$. An infinite volume auto-model is a probability measure $\mu_\theta$ on $(\Omega,\F)$ that is consistent with the family $\pi_{\theta,\Lambda}$ in the sense that
\begin{equation}\label{gibbsmeasure}\mu_\theta\left(f|\F_{\S\setminus\Lambda}\right)(\cdot)=\int \pi_{\theta,\Lambda}(\cdot,dz)f(z),\;\;\mu_\theta-a.s.,\end{equation}
for any finite subset $\Lambda$  of $\S$ and any bounded measurable function $f:\;(\Omega,\F)\to\rset$. Notice that (\ref{gibbsmeasure}) implies that $\mu_\theta\pi_{\theta,\Lambda}=\mu_\theta$, that is, each probability kernel in the family $\pi_\theta$ is invariant with respect to $\mu_\theta$. The probability measure $\mu_\theta$ is an example of a Gibbs measure. We call a random variable  $X=\{X_s,\;s\in\S\}$ with distribution $\mu_\theta$ an auto-model random field with distribution $\mu_\theta$ or with conditional specification $\pi_\theta$. It is well known that given a conditional specification $\pi_\theta$, a consistent Gibbs measure does not always exist and when it does, it is not necessarily unique. In the present case, infinite-volume auto-models exist. This follows for example from \cite{georgii88}, Theorem 4.23 (a).

\begin{prop}
Suppose that (\ref{boundBC}) holds and let $\theta:\;\S\times\S\to\rset$ be an infinite matrix such that $\textsf{deg}(s,\theta)<\infty$ for all $s\in\S$. Then the set of probability measure $\mu_\theta$ that satisfies (\ref{gibbsmeasure}) is nonempty.
\end{prop}

\subsection{Consistency: proof of Theorem \ref{thm2}}\label{proofthm2}
The theorem consists in showing that for almost all sample paths, the function $U_n$ epi-converges to $k(\theta_\star;\cdot)$. It is obtained through a slight modification of  Theorem 5.1 of \cite{hess96}.
\subsubsection{Preliminaries}
Let $(\Vset,\dist)$ be a Polish space with metric $\dist$ and Borel sigma-algebra $\B(\Vset)$. Let $\{g,f_n,\;n\geq 1\}$ be a sequence of real-valued functions defined on $\Vset$.  The next proposition states that if $f_n$ can be written as $f_n=g_n+r_n$, where $r_n$ converges to zero in an appropriate sense, then if $g_n$ epi-converge to $g$, so does $f_n$. The proof is simple and is omitted. It also follows as a special case of \cite{dalmaso93}~Proposition 6.20.

\begin{lemma}\label{lem1}
Suppose that $f_n=g_n+r_n$, where $\{g,r_n,g_n,\;n\geq 1\}$ are real-valued functions defined on $\Vset$, such that $g_n$ epi-converges to $g$. Suppose that $|r_n(u)|\leq c(1+\dist(u,0))\alpha_n$ for all $u\in \Vset$ and for some finite constant $c$,  where $\alpha_n\to 0$. Then $f_n$ epi-converges to $g$.
\end{lemma}

For a real-valued function $f$ on $\Vset$, $k$ integer, we define its Lipschitz approximation of order $k$ as $f^{(k)}(u)\eqdef \inf_{v\in V} \{f(v)+k\dist(u,v)\}$.  The Lipschitz approximation  $f^{(k)}$ is a Lipschitz function (with Lipschitz coefficient $k$). For any $u\in \Vset$ the sequence $\{f^{(k)}(u),\;k\geq 1\}$ is nondecreasing, upper bounded by $f(u)$ and if $f$ itself is Lipschitz on $\Vset$,  $\sup_{k\geq 1} f^{(k)}(u)=f(u)$ (see e.g. \cite{dalmaso93}~Theorem 9.13).

Let $(E,\e)$ be a measurable space. A more useful sigma-algebra to work with is $\hat\e$, the sigma-algebra of universally measurable subsets of $E$ with respect to $\e$. $\hat\e=\cap_{\mu}\e_\mu$ where $\e_\mu$ is the $\mu$-completion of $\e$ with respect to  a $\sigma$-finite measure $\mu$ on $(E,\e)$ and where the intersection is over all $\sigma$-finite measures on $(E,\e)$. 
If $g:\;E\times \Vset\to \rset$ is a function,  $x\in E$ and $k\geq 1$, we denote $g^{(k)}(x,\cdot)$ the Lipschitz approximations of order $k$ of $g(x,\cdot)$. It is known (\cite{hess96}~Proposition 4.4) that if $g$ is $\e\times \B(\Vset)$-measurable, then for any $k\geq 1$, $g^{(k)}$  is $\hat\e\times\B(\Vset)$-measurable. The following result is taken from \cite{hess96}~Proposition 3.4.

\begin{prop}\label{prop10}Let $\{g_n,\;n\geq 1\}$ be a sequence of $\e\times \B(\Vset)$-measurable real-valued functions satisfying the following assumptions. There exist a finite constant $c\in (0,\infty)$, $u_0\in \Vset$, such that $g_n(x,u_0)=0$ for all $n\geq 1$, and
\begin{equation}\label{lipcond1}
\sup_{n\geq 1}\sup_{x\in E}\left|g_n(x,u)-g_n(x,v)\right|\leq cd(u,v) \;\; u,v\in \Vset,\;x\in E.\end{equation}
For $x\in E$, let $\textsf{li}_e g_n(x,\cdot)$ and $\textsf{ls}_e g_n(x,\cdot)$ be the epi-limit inferior and superior of the sequence $\{g_n(x,\cdot),\;n\geq 1\}$ respectively. Then for all $u\in\Vset$,
\[\textsf{li}_e g_n(x,u)=\sup_{k\geq 1}\liminf_{n\to\infty} g_n^{(k)}(x,u),\;\;\mbox{ and }\;\;\textsf{ls}_e g_n(x,u)=\sup_{k\geq 1} \limsup_{n\to\infty} g^{(k)}_n(x,u).\]
\end{prop}

\begin{prop}\label{prop11}Let $\{g_n,\;n\geq 1\}$ be as in Proposition \ref{prop10}, and  let $\{X_k,\;k\geq 1\}$ be a sequence of $E$-valued  random variables defined on some probability space $(\Omega,\mathcal{A},\PP)$. Define
\[h_n(\omega,u)\eqdef \frac{1}{n}\sum_{k=1}^ng_n(X_k(\omega),u),\;n\geq 1,\;\omega\in\Omega,\;u\in \Vset.\]
Suppose that there exists a Lipschitz function $\phi:\;\Vset\to\rset$ and $N\subseteq\Omega$, $\PP(N)=0$ such that for all $\omega\notin N$,
\begin{equation}\label{slln1}\lim_{n\to\infty} h_n(\omega,u) =\phi(u),\;\;\mbox{ for all }u\in \Vset.\end{equation}
Then for all $\omega\notin N$, $\textsf{ls}_e h_n(\omega,u) \leq \phi(u)$, for all $u\in \Vset$, where $\textsf{ls}_e h_n(\omega,\cdot)$ is the epi-limit superior of the function $h_n(\omega,\cdot)$.
\end{prop}
\begin{proof}
This result is part of \cite{hess96}~Theorem 5.1. We give the proof here for completeness.  Fix $\omega\notin N$ and $u\in \Vset$. The Lipschitz property (\ref{lipcond1}) of $g_n$ transfers to $h_n$ and by Proposition \ref{prop10},  $\textsf{ls}_e h_n(\omega,u)=\sup_{k\geq 1} \limsup_{n\to\infty} h_n^{(k)}(\omega,u)$. For any $k\geq 1$ there exists a sequence $\{v_p,\;p\geq 1\}$, $v_p=v_p(u,k)\in \Vset$ such that $\phi^{(k)}(u)=\inf_{p\geq 1}\{\phi(v_p)+kd(u,v_p)\}$. Then
\begin{multline*}
\limsup_{n\to\infty} h_n^{(k)}(\omega,u)=\limsup_{n\to\infty} \inf_{v\in \Vset}\{h_n(\omega,v) +kd(u,v)\}\leq \inf_{v\in \Vset}\limsup_{n\to\infty}\{h_n(\omega,v) +kd(u,v)\}\\
\leq \inf_{p\geq 1}\limsup_{n\to\infty}\{h_n(\omega,v_p) +kd(u,v_p)\}=\inf_{p\geq 1}\{\phi(v_p)+kd(u,v_p)\}=\phi^{(k)}(u).
\end{multline*}
Taking the supremum over $k$ on both side gives the result.
\end{proof}

We now consider the case where $\{X_k,\,k\geq 1\}$ is an i.i.d. sequence.
\begin{prop}\label{prop12}Let $\{g_n,\;n\geq 1\}$ be as in Proposition \ref{prop10}. Suppose that there exists a real-valued, $\e\times \B(\Vset)$-measurable function $g$ such that 
\begin{equation}\label{lipcond2}
\sup_{x\in E}\left|g(x,u)-g(x,v)\right| \leq c\dist(u,v) \;\; u,v\in \Vset,\end{equation}
where $c$ can be taken as in (\ref{lipcond1}), and for any  $(x,u)\in E\times \Vset$
\begin{equation}\label{unifconv}\lim_{n\to \infty} \left|g_n(x,u)- g(x,u)\right|=0 .\end{equation}
Let $\{X_k,\;k\geq 1\}$ be a sequence of $E$-valued, i.i.d. random variables define on some probability space $(\Omega,\mathcal{A},\PP)$. 
 We define, $\phi(u)\eqdef \PE\left(g(X_1,u)\right)$.  Then there exists a $\PP$-negligible subset $N$ of $\Omega$ such that for any $u\in \Vset$ and $\omega\in\Omega\setminus N$, 
\[h_n(\omega,u)\;\;\mbox{epi-converges to } \;\;\phi(u),\;\mbox{ as }\;n\to\infty\]
where $h_n(\omega,u)\eqdef n^{-1}\sum_{k=1}^ng_n(X_k(\omega),u)$.
\end{prop}
\begin{proof}
Notice that we can assume without any loss of generality that the constant $c$ in (\ref{lipcond1}) and (\ref{lipcond2}) is smaller than $1$. Otherwise simply divide $g$ and $g_n$ by $2c$, say. It follows from (\ref{lipcond1}) that
\begin{equation}\label{unifboundgn}
\sup_{n\geq 1}\sup_{x\in E}|g_n(x,u)|\leq cd(u,u_0).\end{equation}
This implies that $g_n$ is bounded in $x$ and that $\phi_n(u)\eqdef \PE\left(g_n(X_1,u)\right)$ is well-defined  and is uniformly bounded in $n$ for each $u$. Now, since $g_n(x,u)$ converges pointwise to $g(x,u)$, we can then apply the  Lebesgue dominated convergence theorem to conclude that $\phi_n(u)\to \phi(u)$ for each $u\in \Vset$. (\ref{lipcond2}) implies also that $\phi$ is Lipschitz.

Furthermore, by the  law of large numbers for arrays of independent random variables, for each $u\in \Vset$, there exists a measurable set $N_1(u)\subseteq\Omega$, $\PP(N_1(u))=0$ such that  for all $\omega\notin N_1(u)$, $\frac{1}{n}\sum_{k=1}^n\left(g_n(X_k(\omega),u)-\phi_n(u)\right)$ converges to zero. Since $\phi_n(u)\to \phi(u)$ and using (\ref{lipcond1}) and the Polish assumption, we conclude that there exists $N_1\subseteq\Omega$, $\PP(N_1)=0$ such that for all $\omega\notin N_1$,
$\lim_{n\to\infty} h_n(\omega,u) =\phi(u),\;\;\mbox{ for all }u\in V$. By Proposition \ref{prop11}, we obtain for all $\omega\notin N_1$, $\textsf{ls}_eh_n(\omega,u)\leq \phi(u)$, for all $u\in \Vset$.  

We will now show that there exists $N_2\subseteq\Omega$, $\PP(N_2)=0$ such that for all $\omega\notin N_2$,
\begin{equation}\label{slln2}
\liminf_{n\to\infty}h^{(k)}_n(\omega,u)\geq \PE\left(g^{(k)}(X_1,u)\right), \;\;\mbox{ for all }\;\;u\in \Vset,\;k\geq 1.\end{equation}
By Proposition \ref{prop10}, we can then deduce that for all $\omega\in \Omega\setminus N_2$, and for all $u\in \Vset$, $\textsf{li}_e h_n(\omega,u) \geq \sup_{k\geq 1}\PE\left(g^{(k)}(X_1,u)\right)=\lim_{k\to\infty}\PE\left(g^{(k)}(X_1,u)\right)= \phi(u)$, by dominated convergence. And the result will be proved.

Let us show that (\ref{slln2}) holds. Fix $u\in \Vset$, $k\geq 1$ integer. Notice that $g_n^{(k)}(x,u)\leq g_n(x,u)$ and given the boundedness of $g_n$, we apply the law of large numbers  to $g_n^{(k)}$ to conclude that there exists $N_2\subseteq\Omega$, $\PP(N_2)=0$ such that for all $\omega\notin N_2$,
\begin{multline}\label{slln3} \liminf_{n\to\infty} h^{(k)}_n(\omega,u)\geq\liminf_{n\to\infty} \frac{1}{n}\sum_{i=1}^n g_n^{(k)}(X_i(\omega),u) \\
=\lim_{n\to\infty}\frac{1}{n}\sum_{i=1}^n\left(g_n^{(k)}(X_i(\omega),u)-\PE\left(g_n^{(k)}(X_1,u)\right)\right) +  \liminf_{n\to\infty}\PE\left(g_n^{(k)}(X_1,u)\right) \\
= \liminf_{n\to\infty}\PE\left(g_n^{(k)}(X_1,u)\right).\end{multline}
We obtain as a consequence of  (\ref{lipcond1}) that  $|g_n(x,u)|\leq c\dist(u,u_0)$, for all $n,x,u$. Consequently,  for any $v\in V$, and $k\geq 1$, $g_n(x,v)+k \dist(u,v)\geq -c \dist(v,u_0)+k\dist(u,v)\geq -c \dist(u,u_0)$. This shows that there exists a finite constant $C(u)$ (for example $c\dist(u,u_0)$) such that $g_n^{(k)}(x,u)+C(u)\geq 0$ for all $x\in E$. By Fatou's lemma, we deduce that
\begin{equation}\label{slln4}
\liminf_{n\to\infty}\PE\left(g_n^{(k)}(X_1,u)\right)\geq \PE\left(\liminf_{n\to\infty}g_n^{(k)}(X_1,u)\right).\end{equation}

Fix $x\in E$. Given $\epsilon>0$, we can find $v_0=v_0(x,u,n,k,\epsilon)\in \Vset$ such that $g_n^{(k)}(x,u)> g_n(x,v_0)+k\dist(u,v_0)-\epsilon$. Because of (\ref{lipcond1}), $v_0\in B(u,\epsilon/(1-c))$, where $B(x,r)$ is the ball of center $x$ and radius $r$. Indeed, if $\dist(u,v)>\epsilon/(1-c)$, then $g_n(x,v)+k\dist(u,v)\geq  g_n(x,u)+(k-c) \dist(u,v)\geq g_n(x,u)+\epsilon\geq g_n^{(k)}(x,u)+\epsilon$. Thus 
\begin{multline*}
g_n^{(k)}(x,u)> g_n(x,v_0)+k\dist(u,v_0)-\epsilon = g(x,v_0)+k\dist(u,v_0) + \left(g(x,u)-g(x,v_0)\right) \\
+ \left(g_n(x,u)-g(x,u)\right) + \left(g_n(x,v_0)-g_n(x,u)\right) -\epsilon\\
\geq g^{(k)}(x,u)  +  \left(g_n(x,u)-g(x,u)\right) -(1-c)^{-1}\epsilon.\end{multline*}
Taking the $\liminf$ as $n\to\infty$ on both side and letting $\epsilon\to 0$ together with (\ref{unifconv}) gives $\liminf_{n\to\infty}g_n^{(k)}(x,u)\geq g^{(k)}(x,u)$. Combining that with (\ref{slln3}) and (\ref{slln4})  yields (\ref{slln2}).
\end{proof}

\subsubsection{Proof of Theorem \ref{thm2}}\label{proofthm2}
Write $U_n(\theta)=\bar U_n(\theta) + r_n(\theta)$, where 
\[r_n(\theta)=n^{-1}\sum_{(s,\ell)\in\underline{D}_n^2}\left(q_{\lambda_n}(|\theta_\star(s,\ell)+\theta(s,\ell)|)-q_{\lambda_n}(|\theta_\star(s,\ell)|)\right).\]
Set $\mathcal{L}_0^{(n)}(\theta)\eqdef\{(s,\ell)\in\underline{D}_n^2:\; \theta(s,\ell)\neq0\}$. By  the Mean Value Theorem and A\ref{A0},
\begin{multline*}
\left|r_n(\theta)\right|=n^{-1}\lambda_n\left|\sum_{(s,\ell)\in\mathcal{L}_0^{(n)}(\theta)}\lambda_n^{-1}\left(q_{\lambda_n}(|\theta_\star(s,\ell)+\theta(s,\ell)|)-q_{\lambda_n}(|\theta_\star(s,\ell)|)\right)\right|\\
\leq n^{-1}\lambda_n \sum_{(s,\ell)\in\mathcal{L}_0^{(n)}(\theta)} c \left(|\theta_\star(s,\ell)+\theta(s,\ell)|-|\theta_\star(s,\ell)|\right)\leq cn^{-1}\lambda_n\|\theta\|_1,\end{multline*}
for some finite constant $c$. Thus Lemma \ref{lem1} applies and it is enough to show that almost surely, $\bar U_n$ epi-converges to $k(\theta_\star;\cdot)$.  

To do so, we apply Proposition \ref{prop12}. Take $E=\Xset^\S$ with generic element $x=\{x(s),\;s\in \S\}$ and $V=\M_1$ and
\[g_n(x,\theta)=\sum_{s\in D_n}-\log\left(\frac{f_{\theta_\star+\theta}^{(s)}(x_s\vert x_{\partial_n s})}{f_{\theta_\star}^{(s)}(x_s\vert x_{\partial_n s})}\right).\]
The limiting function $g$ is given by 
\[g(x,\theta)=\sum_{s\in \S}-\log\left(\frac{f_{\theta_\star+\theta}^{(s)}(x_s\vert x_{\S\setminus\{s\}})}{f_{\theta_\star}^{(s)}(x_s\vert x_{\S\setminus\{s\}})}\right).\]
We have seen earlier that as a consequence of (Equation \ref{Id:PS}), $|g(x,\theta)|<\infty$. It is clear that $g_n$ is a real-valued normal integrand, $g_n(x,\textsf{0})=0$ and it follows also  from (\ref{Id:PS}) and (\ref{boundBC}) that
\[\sup_{x\in E}\left|g(x,\theta)-g(x,\theta')\right| + \sup_{n\geq 1}\sup_{x\in E} \left|g_n(x,\theta)-g_n(x,\theta')\right|\leq C\|\theta-\theta'\|_1,\]
for some finite constant $C$. Thus (\ref{lipcond1}) and (\ref{lipcond2}) hold.  It remains to show (\ref{unifconv}).

Consider $x\in\Xset^\S$ and $\theta\in\M_1$. Since $\|\theta\|_1<\infty$, for any $\epsilon>0$, there exists a finite subset $\Lambda_\epsilon\subset\S$ such that $\sum_{(u,v)\notin\underline{\Lambda_\epsilon}^2}|\theta(u,v)|<\epsilon$. We have
\begin{multline}\label{proofthm1eq1}
\left|g_n(x,\theta)-g(x,\theta)\right|\leq \sum_{s\in D_n}\left|-\log\left(\frac{f_{\theta_\star + \theta}^{(s)}(x_s\vert x_{\partial_n s})}{f_{\theta_\star}^{(s)}(x_s\vert x_{\partial_n s})}\right) + \log\left(\frac{f_{\theta_\star + \theta}^{(s)}(x_s\vert x_{\S\setminus\{s\}})}{f_{\theta_\star}^{(s)}(x_s\vert x_{\S\setminus\{s\}})}\right)\right|\\
+\sum_{s\in \S\setminus D_n}\left|\log\left(\frac{f_{\theta_\star + \theta}^{(s)}(x_s\vert x_{\S\setminus\{s\}})}{f_{\theta_\star}^{(s)}(x_s\vert x_{\S\setminus\{s\}})}\right)\right|.\end{multline}
We first deal with the second term on the right-hand side of (\ref{proofthm1eq1}). Fix $\epsilon>0$. Take $n$ large enough such that $\Lambda_\epsilon\subseteq D_n$. Then using again (\ref{Id:PS}) and (\ref{boundBC}), we have
\[\sum_{s\in \S\setminus D_n}\left|\log\left(\frac{f_{\theta_\star + \theta}^{(s)}(x_s\vert x_{\S\setminus\{s\}})}{f_{\theta_\star}^{(s)}(x_s\vert x_{\S\setminus\{s\}})}\right)\right|\leq C\sum_{s\in \S\setminus D_n}\sum_{\ell\in \S}|\theta(s,\ell)| \leq C\epsilon.\]
The first term is obtained from
\begin{multline*}
\left(\log f_{\theta_\star+\theta}^{(s)}(x_s\vert x_{\S\setminus\{s\}}) - \log f_{\theta_\star}^{(s)}(x_s\vert x_{\S\setminus\{s\}})\right) - \left(\log f_{\theta_\star+\theta}^{(s)}(x_s\vert x_{\partial_n s})-\log f_{\theta_\star}^{(s)}(x_s\vert x_{\partial_n s})\right)\\
=\sum_{\ell\in\S}\theta(s,\ell)B(x_s,x_\ell) -\int_0^1 dt\left\{\int_\Xset \sum_{\ell\in \S}\theta(s,\ell)\bar B_{s,\ell}(u,x_\ell)  f^{(s)}_{\theta_\star+t\theta}(u\vert x_{\S\setminus\{s\}})\rho(du)\right\}\\
-\sum_{\ell\in D_n}\theta(s,\ell)B(x_s,x_\ell)  + \int_0^1 dt\left\{\int_\Xset \sum_{\ell\in D_n}\theta(s,\ell)\bar B_{s,\ell}(u,x_\ell) f^{(s)}_{\theta_\star +t\theta}(u\vert x_{\partial_n s})\rho(du)\right\},
\end{multline*}
where $\bar B_{s,\ell}(x,y)=B_0(x)$ if $\ell=s$ and $\bar B_{s,\ell}(x,y)=B(x,y)$ otherwise. The above equality follows from Lemma \ref{lemlip}. We use this to conclude that there exists a finite constant $C$ such that 
\begin{multline*}
\left|-\log\left(\frac{f_{\theta_\star + \theta}^{(s)}(x_s\vert x_{\partial_n s})}{f_{\theta_\star}^{(s)}(x_s\vert x_{\partial_n s})}\right) + \log\left(\frac{f_{\theta_\star + \theta}^{(s)}(x_s\vert x_{\S\setminus\{s\}})}{f_{\theta_\star}^{(s)}(x_s\vert x_{\S\setminus\{s\}})}\right)\right|\leq C\sum_{\ell\in \S\setminus D_n}|\theta(s,\ell)| \\
+\sum_{\ell\in D_n}|\theta(s,\ell)|\int_0^1 dt\left| \int_\Xset \bar B_{s,\ell}(u,x_\ell)\left(f^{(s)}_{\theta_\star+t\theta}(u\vert x_{\S\setminus\{s\}})-f^{(s)}_{\theta_\star+t\theta}(u\vert x_{\partial_n s})\right)\rho(du)\right|.\end{multline*}
Taking the sum over $s\in D_n=\Lambda_\epsilon\cup D_n\setminus\Lambda_\epsilon$  we get
\begin{multline*}
\sum_{s\in D_n}\left|-\log\left(\frac{f_{\theta_\star + \theta}^{(s)}(x_s\vert x_{\partial_n s})}{f_{\theta_\star}^{(s)}(x_s\vert x_{\partial_n s})}\right) + \log\left(\frac{f_{\theta_\star + \theta}^{(s)}(x_s\vert x_{\S\setminus\{s\}})}{f_{\theta_\star}^{(s)}(x_s\vert x_{\S\setminus\{s\}})}\right)\right|\leq C\sum_{s\in D_n}\sum_{\ell\in \S\setminus D_n}|\theta(s,\ell)|  \\
+\sum_{s\in D_n}\sum_{\ell\in D_n}|\theta(s,\ell)|\int_0^1 dt \left|\int_\Xset \bar B_{s,\ell}(u,x_\ell)\left(f^{(s)}_{\theta_\star+t\theta}(u\vert x_{\S\setminus\{s\}})-f^{(s)}_{\theta_\star+t\theta}(u\vert x_{\partial_n s})\right)\rho(du)\right|\\
\leq C\epsilon + \sum_{s\in \Lambda_\epsilon}\sum_{\ell\in\S}|\theta(s,\ell)|\int_0^1 dt \left|\int_\Xset \bar B_{s,\ell}(u,x_\ell)\left(f^{(s)}_{\theta_\star+t\theta}(u\vert x_{\S\setminus\{s\}})-f^{(s)}_{\theta_\star+t\theta}(u\vert x_{\partial_n s})\right)\rho(du)\right|.\end{multline*}
For each $s$, the inner sum in the last term converges to 0 as $n\to\infty$. Since $\Lambda_\epsilon$ is finite, we conclude that 
\[\lim_{n\to\infty}\sum_{s\in D_n}\left|-\log\left(\frac{f_{\theta_\star + \theta}^{(s)}(x_s\vert x_{\partial_n s})}{f_{\theta_\star}^{(s)}(x_s\vert x_{\partial_n s})}\right) + \log\left(\frac{f_{\theta_\star + \theta}^{(s)}(x_s\vert x_{\S\setminus\{s\}})}{f_{\theta_\star}^{(s)}(x_s\vert x_{\S\setminus\{s\}})}\right)\right|\leq C\epsilon.\]
It follows that (\ref{unifconv}) holds. Finally by conditioning on $X_{\S\setminus\{s\}}$, we notice that
\begin{multline*}
\PE_{\theta_\star}\left[-\log\left(\frac{f_{\theta_\star+\theta}^{(s)}(X_s\vert X_{\S\setminus\{s\}})}{f_{\theta_\star}^{(s)}(X_s\vert X_{\S\setminus\{s\}})}\right)\right]=\PE_{\theta_\star}\left(\int -\log\left(\frac{f_{\theta_\star+\theta}^{(s)}(u\vert X_{\S\setminus\{s\}})}{f_{\theta_\star}^{(s)}(u\vert X_{\S\setminus\{s\}})}\right)f_{\theta_\star}^{(s)}(u\vert X_{\S\setminus\{s\}})du\right)\\
=k^{(s)}(\theta_\star,\theta).\end{multline*}
The theorem is proved.

\begin{flushright}
$\square$
\end{flushright}

\subsubsection{Proof of Corollary \ref{coro1}}\label{proofcoro1}
Let us first show that $k(\theta_\star,\cdot)$ admits a unique minimum at $\textsf{0}$.  Since $k^{(s)}(\theta_\star,\cdot)$ is nonnegative, $k(\theta_\star,\theta)=0$ implies that $k^{(s)}(\theta_\star,\theta)=0$ for all $s\in\S$. We use Lemma \ref{lemlip} to write
\begin{multline*}
-\log f^{(s)}_{\theta_\star+\theta}(X_s\vert X_{\S\setminus\{s\}})+\log f^{(s)}_{\theta_\star}(X_s\vert X_{\S\setminus\{s\}})=\\
-\sum_{\ell\in \S}\theta(s,\ell)\left(B(X_s,X_\ell) -\int_\Xset \bar B_{s,\ell}(u,X_\ell)f^{(s)}_{\theta_\star}(u\vert X_{\S\setminus\{s\}})du\right) \\
+\int_\Xset \sum_{\ell\in\S}\theta(s,\ell)\bar B_{s,\ell}(u,X_\ell)\int_0^1dt\left(f^{(s)}_{\theta_\star+t\theta}(u\vert X_{\S\setminus\{s\}})-f^{(s)}_{\theta_\star}(u\vert X_{\S\setminus\{s\}})\right)du.\end{multline*}
 Taking the expectation on both side and using Lemma \ref{lemlip} again yields
\begin{multline*}
k^{(s)}(\theta_\star,\theta)=\int_0^1tdt\int_0^1d\tau\PE_\star\left[\textsf{Var}_{\theta_\star+t\tau\theta}\left(\sum_{\ell\in \S}\theta(s,\ell)\bar B_{s,\ell}(X_s,X_\ell)\vert X_{\S\setminus\{s\}}\right)\right]\\
=\int_0^1tdt\int_0^1d\tau\sum_{\ell,\ell'\in\S}\theta(s,\ell)\theta(s,\ell')\rho^{(s)}_{\theta_\star+t\tau\theta}(\ell,\ell').
\end{multline*}
Since $\rho^{(s)}_\theta$ is positive definite, $k^{(s)}(\theta_\star,\theta)=0$ if and only if $\theta(s,\ell)=0$ for all $\ell\in\S$. 

Now, let $\epsilon>0$. By tightness, there exists a compact subset $\compact$ of $\M_1$ such that $\sup_{n\geq 1} \cPP_\star\left((\hat\theta_n-\theta_\star^{(n)})\notin \compact\right)\leq \epsilon$. Therefore 
\begin{multline*}
\cPP_\star\left(\|\hat\theta_n-\theta_\star^{(n)}\|_1>\epsilon\right)\leq \epsilon + \cPP_\star\left((\hat\theta_n-\theta_\star^{(n)})\in \compact,\;\;\|\hat\theta_n-\theta_\star^{(n)}\|_1>\epsilon\right)\\
\epsilon +  \cPP_\star\left(\cup_{m\geq n}\left\{(\hat\theta_m-\theta_\star^{(m)})\in \compact,\;\;\|\hat\theta_m-\theta_\star^{(m)}\|_1>\epsilon\right\}\right).\end{multline*}
We conclude that 
\[\lim_{n\to\infty}\cPP_\star\left(\|\hat\theta_n-\theta_\star^{(n)}\|_1>\epsilon\right)\leq \epsilon + \cPP_\star\left(\left\{(\hat\theta_n-\theta_\star^{(n)})\in \compact,\;\;\|\hat\theta_n-\theta_\star^{(n)}\|_1>\epsilon\right\}\;\;\textsf{i.o.}\right).\]
Corollary 7.20 of \cite{dalmaso93} and Theorem \ref{thm2} imply that the probability on the rhs is zero. This ends the proof.
 
\begin{flushright}
$\square$
\end{flushright}

\subsection{Rate of convergence: proof of Theorem \ref{thm3}}\label{proofthm3}

The proof of the theorem is adapted from Chapter 3.4 of \cite{vaartetwellner96} of the rate of convergence of M-estimators. Fix $\epsilon>0$. 
Let $C,c_0<\infty$ such that $\|B_0\|_\infty+\|B\|_\infty\leq C$ and $\sup_{\lambda>0}\sup_{x>0}q_\lambda'(x)\leq c_0$. Under the stated assumptions, $\alpha_n^{-1}r_na_n^{1/2}\lambda_nn^{-1}=O(1)$, as $n\to\infty$. Therefore, we can take  $M>1$ large enough so that for all $n\geq 1$,
\begin{equation}\label{choiceM}
16c_0\alpha_n^{-1} r_na_n^{1/2}\lambda_nn^{-1} \leq 2^M,\;\;\mbox{ and }\;\;\;16\sum_{j\geq M}2^{-j}\leq \epsilon.\end{equation}
For $j\geq 1$, define $\Theta_{n,j}=\{\theta\in\M^{(n)}(a_n,\tau_n):\; 2^{j-1}<r_n\|\theta\|_2\leq 2^j\}$. Clearly we have,
\[
\left\{r_n\|\hat\theta_n-\theta_\star^{(n)}\|_2>2^M\right\}
\subseteq \bigcup_{j\geq M}\left\{\hat\theta_n-\theta_\star^{(n)}\in\Theta_{n,j}\right\}.\]
On the other hand, since $U_n$ admits a minimum at $\hat\theta_n-\theta_\star^{(n)}$, almost surely, and $U_n(\textsf{0})=0$, it follows that $\{\hat\theta_n-\theta_\star^{(n)}\in\Theta_{n,j}\}\subseteq\{\inf_{\theta\in\Theta_{n,j}}U_n(\theta)\leq 0\}$. 
We conclude that
\begin{equation}\label{proofthm3eq1}
\cPP_\star\left(r_n\|\hat\theta_n-\theta_\star^{(n)}\|_2>2^M\right)\leq \sum_{j\geq M}\cPP_\star\left(\inf_{\theta\in\Theta_{n,j}}U_n(\theta)\leq 0\right).\end{equation}
We  recall that
\begin{multline*}
U_n(\theta)=n^{-1}\left(\bar\ell_n(\theta_\star^{(n)})-\bar\ell_n(\theta_\star^{(n)}+\theta)\right)\\
+n^{-1}\sum_{(s,\ell)\in\underline{D}_n^2}\left(q_{\lambda_n}(|\theta_\star(s,\ell)+\theta(s,\ell)|)-q_{\lambda_n}(|\theta_\star(s,\ell)|)\right).
\end{multline*}
Set $\mathcal{L}_0^{(n)}(\theta)=\{(s,\ell)\in \underline{D}_n^2:\;\theta(s,\ell)\neq 0\}$. For $\theta\in\Theta_{n,j}$, and using the mean value theorem and A\ref{A0},
\begin{multline}\label{boundpenalty}
n^{-1}\sum_{(s,\ell)\in\underline{D}_n^2}\left(q_{\lambda_n}(|\theta_\star(s,\ell)+\theta(s,\ell)|)-q_{\lambda_n}(|\theta_\star(s,\ell)|)\right)\\
=n^{-1}\lambda_n\sum_{(s,\ell)\in\mathcal{L}_0^{(n)}(\theta)} \lambda_n^{-1}\left(q_{\lambda_n}(|\theta_\star(s,\ell)+\theta(s,\ell)|)-q_{\lambda_n}(|\theta_\star(s,\ell)|)\right)
\leq c_0n^{-1}\lambda_n\sum_{(s,\ell)\in\mathcal{L}_0^{(n)}(\theta)}|\theta(s,\ell)|\\
\leq  c_0n^{-1}\lambda_n a_n^{1/2}\|\theta\|_2\leq c_0n^{-1}\lambda_n a_n^{1/2} 2^jr_n^{-1}.\end{multline}
Now, for $\theta\in\M^{(n)}$, $n^{-1}\left(\bar\ell_n(\theta_\star^{(n)})-\bar\ell_n(\theta_\star^{(n)}+\theta)\right)=n^{-1}\sum_{i=1}^n m_{n,\theta}(X^{(i)})=n^{-1}\sum_{i=1}^n \bar m_{n,\theta}(X^{(i)}) +M_n(\theta)$,  where $\bar m_{n,\theta}(x)=m_{n,\theta}(x)-M_n(\theta)$, with $m_{n,\theta}$ as in (\ref{defmn}), and
\begin{multline*} M_n(\theta)=\sum_{s\in D_n}\sum_{\ell\in D_n}\theta(s,\ell)\PE_\star\left[\int_0^1dt\int \bar B_{s,\ell}(u,X_\ell)\left(f_{\theta_\star^{(n)}}^{(s)}(u\vert X_{\partial_n s})-f_{\theta^{(n)}_\star+t\theta}^{(s)}(u\vert X_{\partial_n s})\right)du\right]\\
= \int_0^1tdt\int_0^1 d\tau \sum_{s\in D_n}\PE_\star\left[\textsf{Var}_{\theta_\star^{(n)}+ t\tau\theta}\left(\sum_{\ell\in D_n}\theta(s,\ell)B(X_s,X_\ell)\vert X_{\partial_n s}\right)\right]\geq \frac{\alpha_n}{2}\|\theta\|_2^2, \end{multline*}
using Lemma \ref{lemlip} and A\ref{B2}.  
Notice that the first part of (\ref{choiceM}) implies that $\frac{\alpha_n}{4}2^{2(j-1)}r_n^{-2}\geq c_02^{j}r_n^{-1}a_n^{1/2}\lambda_nn^{-1}$ whenever $j\geq M$. Therefore, using (\ref{boundpenalty}),
\[
\inf_{\theta\in\Theta_{n,j}}\left\{U_n(\theta)\right\}\geq \inf_{\theta\in\Theta_{n,j}}\left\{n^{-1}\sum_{i=1}^n \bar m_{n,\theta}(X^{(i)})\right\} + \frac{\alpha_n}{4} 2^{2(j-1)}r_n^{-2},\]
and (\ref{proofthm3eq1}) becomes
\begin{multline}\label{proofthm3eq2}
\cPP_\star\left(r_n\|\hat\theta_n-\theta_\star^{(n)}\|_1>2^M\right)\leq
\sum_{j\geq M}\cPP_\star\left(\sup_{\theta\in\Theta_{n,j}}\left|n^{-1/2}\sum_{i=1}^n \bar m_{n,\theta}(X^{(i)})\right|\geq \frac{\alpha_n}{16}\frac{\sqrt{n}2^{2j}}{r_n^{2}}\right)\\
\leq \frac{16}{\alpha_n}\frac{r_n^2}{\sqrt{n}}\sum_{j\geq M}2^{-2j}\left[\cPE_\star\left(\|\G_n\|_{\F_{n,j}}\right) + \sup_{\theta\in\Theta_{n,j}}\sqrt{n}\left|\PE_{\theta_\star}(m_{n,\theta}(X))-M_n(\theta)\right|\right],\end{multline}
where $\G_n$ is the empirical process associated to the family $\F_{n,j}=\F_{n,2^jr_n^{-1}}$: for $f\in\F_{n,j}$, $\G_n(f)=n^{-1/2}\sum_{i=1}^n \left(f(X^{(i)})-\PE_\star(f(X^{(1)})\right)$. And $\|\G_n\|_{\F_{n,j}}\eqdef \sup_{\theta\in\F_{n,j}}|\G_n(f)|$. 

Using the fact that $\PE_\star\left(B(X_s,X_\ell)\vert X_{\S\setminus\{s\}}\right)=\int \bar B_{s,\ell}(u,X_\ell)f_{\theta_\star}^{(s)}(u\vert X_{\S\setminus\{s\}})du$, $\mu_\star$-a.s., together with Lemma \ref{lemlip}, we have 
\begin{multline}\label{boundfin1}\left|\PE_{\theta_\star}(m_{n,\theta}(X))-M_n(\theta)\right|\\
=\left|\sum_{s\in D_n}\sum_{\ell\in D_n}\theta(s,\ell)\PE_\star\left[\int_\Xset \bar B_{s,\ell}(X_s,X_\ell)\left(f_{\theta_\star}^{(s)}(u\vert X_{\partial s})-f_{\theta_\star^{(n)}}^{(s)}(u\vert X_{\partial_n s})\right)du\right]\right|\\
\leq \sum_{s\in \Delta_n^{(c)}}\left(\sum_{\ell\in D_n}|\theta(s,\ell)|\right)\left(\sum_{\ell\in \partial s\setminus D_n}|\theta_\star(s,\ell)|\right)\leq b_n \tau_n \|\theta\|_2\leq b_n \tau_n 2^{j}r_n^{-1},\end{multline}
where $\Delta_n^{(c)}= \{s\in D_n:\; \partial s\setminus D_n \neq \emptyset\}$. 
 Notice that 
\[
|m_{n,\theta}(x)|\leq 2C\|\theta\|_{1}\leq c \beta_{n,j},\;\;\mbox{ for all }\;\;\theta\in \Theta_{n,j},\]
where $\beta_{n,j}=a_n^{1/2}2^jr_n^{-1}$, for some finite  constant $c$.  Also for $\theta\in\Theta_{n,j}$, the second part of B\ref{B2} yields
\[\PE_\star^{1/2}\left(m^2_{n,\theta}(X)\right) \leq \alpha_n' \|\theta\|_2\leq \delta_{n,j},\]
where $\delta_{n,j}=\alpha_n'2^jr_n^{-1}$ for some finite constant $c$. By Lemma 3.4.2 of \cite{vaartetwellner96},
\[\cPE_\star\left(\|\G_n\|_{\F_{n,j}}\right)\leq cJ_{[]}\left(\delta_{n,j},\F_{n,j},L^2(\mu_{\theta_\star})\right)\left(1+\frac{c\beta_{n,j}}{\sqrt{n}\delta_{n,j}^2}J_{[]}\left(\delta_{n,j},\F_{n,j},L^2(\mu_{\theta_\star})\right)\right),\]
for some finite constant $c$, where $J_{[]}\left(\delta_{n,j},\F_{n,j},L^2(\mu_{\theta_\star})\right)$ is the bracketing integral of the family $\F_{n,j}$ defined as
\[J_{[]}\left(\delta_{n,j},\F_{n,j},L^2(\mu_{\theta_\star})\right)=\int_0^{\delta_{n,j}} \sqrt{1+\log N_{[]}\left(\epsilon,\F_{n,j},L^2(\mu_{\theta_\star})\right)}d\epsilon.\]

For any $\theta,\theta'\in\Theta_{n,j}$, $\left|m_{n,\theta}(x)-m_{n,\theta'}(x)\right| \leq c\|\theta-\theta'\|_{1}\leq 2ca_n^{1/2}\|\theta-\theta'\|_2$, for all $x\in\Xset^\infty$. 
This Lipschitz property of the family $\F_{n,j}$, Theorem 2.7.11 of \cite{vaartetwellner96}  and  (\ref{coveringnumber}) imply that
\begin{multline*}
J_{[]}\left(\delta_{n,j},\F_{n,j},L^2(\mu_{\theta_\star})\right)\leq  c a_n^{1/2}\int_0^{\delta_{n,j}a_n^{-1/2}/4c} \sqrt{1+\log N\left(\epsilon,\Theta_{n,j},\|\cdot\|_2\right)}d\epsilon\\
\leq c\delta_{n,j}\sqrt{a_n\log\left(\frac{p_n}{a_n}\right)},\end{multline*}
for some finite constant $c$. Under the assumption $a_n\sqrt{\log p_n}=O(\alpha_n'\sqrt{n})$, we obtain  that
$n^{-1/2}\beta_{n,j}\delta_{n,j}^{-2}J_{[]}\left(\delta_{n,j},\F_{n,j},L^2(\mu_{\theta_\star})\right)\leq c n^{-1/2}a_n\sqrt{\log p_n}/\alpha_n'=O(1)$. As the result, $\cPE_\star\left(\|\G_n\|_{\F_{n,j}}\right)\leq c\delta_{n,j}\sqrt{a_n\log p_n}$. Combined with (\ref{boundfin1}) and (\ref{proofthm3eq2}) and the expression or $r_n$, it follows that $\cPP_\star\left(r_n\|\hat\theta_n-\theta_\star^{(n)}\|_1>2^M\right)\leq \epsilon c$ for some universal constant $c$. Since $\epsilon>0$ is arbitrary, the theorem follows.

\begin{flushright}
$\square$
\end{flushright}

\subsection{A comparison lemma}
\begin{lemma}\label{lemlip}
Let $(\Yset,\A,\nu)$ be a measure space where $\nu$ is a finite
measure. Let $g_1,g_2,f_1,f_2:\;\Yset\to\rset$ be bounded measurable
functions. For $i\in\{1,2\}$, define $Z_i=\int e^{g_i(y)}\nu(dy)$. For
$t\in [0,1]$, let $\bar g_t(\cdot)=tg_2(\cdot)+(1-t)g_1(\cdot)$ and
$Z_t=\int_\Yset e^{\bar g_t(y)}\nu(dy)$. Let $\bar f_t:\;\Yset\to\rset$ be such
that $\bar f_0=f_1$ and $\bar f_1=f_2$. Suppose that $\frac{d}{dt}\bar f_t(y)$ exists for
$\nu$-almost all $y\in\Yset$ and
$\sup_{t\in[0,1],y\in\Yset}|\frac{d}{dt}\bar f_t(y)|<\infty$.
Then
\begin{multline}
\int f_2(y)e^{g_2(y)}Z_{g_2}^{-1}\nu(dy)-\int
f_1(y)e^{g_1(y)}Z_{g_1}^{-1}\nu(dy)=\int_0^1dt\int_\Yset\left(\frac{d}{dt}\bar
f_t(y)\right)e^{\bar g_t(y)}Z_t^{-1}\nu(dy)\\
+ \int_0^1dt\textsf{Cov}_t\left(\bar f_t(X),(g_2-g_1)(X)\right),\end{multline}
where $\textsf{Cov}_t(U_1(X),U_2(X))$ is the covariance between
$U_1(X)$ and $U_2(X)$ assuming that $X\sim e^{\bar g_t(y)}Z_t^{-1}$.
\end{lemma}
\begin{proof}
Under the stated assumptions, the function $t\to\int_\Yset
\bar f_t(y)e^{\bar g_t(y)}Z_t^{-1}\nu(dy)$ is differentiable under the integral sign and we have:
\[\int f_2(y)e^{g_2(y)}Z_{g_2}^{-1}\nu(dy)-\int
f_1(y)e^{g_1(y)}Z_{g_1}^{-1}\nu(dy)=\int_0^1\frac{d}{dt}\left(
  \int_\Yset \bar f_t(y)e^{\bar g_t(y)}Z_{t}^{-1}\nu(dy)\right)dt.\]
The identity follows by carrying the differentiation under the integral sign.
\end{proof}

With the choice $\bar f_t(y)=t f_2(y)+ (1-t)f_1(y)$, we get
\begin{multline}\label{conseq1lemlip}
\left|\int f_2(y)e^{g_2(y)}Z_{g_2}^{-1}\nu(dy)-\int
f_1(y)e^{g_1(y)}Z_{g_1}^{-1}\nu(dy)\right|\\
\leq
\|f_2-f_1\|_\infty+2(\|f_1\|_\infty+\|f_2\|_\infty)\|g_2-g_1\|_\infty.
\end{multline}

We will also need the following particular case. For bounded measurable function
$h_1,h_2:\;\Yset\to\rset$, we can take $f_i(y)\equiv \log\int
e^{h_i(u)}\nu(du)$, $i=1,2$, $\bar f_t(y)\equiv\log \int e^{t h_2(u)+(1-t)
  h_1(u)}\nu(du)$, and $g_1=g_2$ in the lemma and get:
\begin{multline*}\label{Id:PS}
\log \int e^{h_2(y)}\nu(dy)-\log\int
e^{h_1(y)}\nu(dy)=\int_0^1dt\left(\frac{d}{dt}\bar f_t\right)\\
=\int_0^1\int_{\Yset}\left(h_2(y)-h_1(y)\right)\frac{e^{t h_2(u)+(1-t)
  h_1(u)}}{\int e^{t h_2(u)+(1-t)
  h_1(u)}\nu(du)}\nu(dy).\end{multline*}
In particular,
\begin{equation}\label{Id:PS}
\left|\log \int e^{h_2(y)}\nu(dy)-\log\int
e^{h_1(y)}\nu(dy)\right|\leq \|h_2-h_1\|_\infty.\end{equation}

\vspace{3.0cm}

{\bf Acknowledgment:} I'm grateful to Lisa Levina, Jian Guo, George Michailidis, and Ji Zhu  for helpful discussions. This work is partly supported by NSF grant DMS 0906631.

\vspace{1cm}
\bibliographystyle{ims}
\bibliography{biblio}

\begin{thebibliography}{20}
\expandafter\ifx\csname natexlab\endcsname\relax\def\natexlab#1{#1}\fi
\expandafter\ifx\csname url\endcsname\relax
  \def\url#1{\texttt{#1}}\fi
\expandafter\ifx\csname urlprefix\endcsname\relax\def\urlprefix{URL }\fi

\bibitem[{Banerjee et~al.(2008)Banerjee, El~Ghaoui and
  d'Aspremont}]{barnejeeetal08}
\textsc{Banerjee, O.}, \textsc{El~Ghaoui, L.} and \textsc{d'Aspremont, A.}
  (2008).
\newblock Model selection through sparse maximum likelihood estimation for
  multivariate {G}aussian or binary data.
\newblock \textit{J. Mach. Learn. Res.} \textbf{9} 485--516.

\bibitem[{Besag(1974)}]{besag74}
\textsc{Besag, J.} (1974).
\newblock Spatial interaction and the statistical analysis of lattice systems.
\newblock \textit{J. Roy. Statist. Soc. Ser. B} \textbf{36} 192--236.
\newblock With discussion by D. R. Cox, A. G. Hawkes, P. Clifford, P. Whittle,
  K. Ord, R. Mead, J. M. Hammersley, and M. S. Bartlett and with a reply by the
  author.

\bibitem[{Bickel and Levina(2008)}]{bickeletlevina08b}
\textsc{Bickel, P.~J.} and \textsc{Levina, E.} (2008).
\newblock Regularized estimation of large covariance matrices.
\newblock \textit{Ann. Statist.} \textbf{36} 199--227.

\bibitem[{Dal~Maso(1993)}]{dalmaso93}
\textsc{Dal~Maso, G.} (1993).
\newblock \textit{An introduction to Gamma-convergence}.
\newblock Progress in nonlinear differential equations and their applications,
  Birkhauser, Boston.

\bibitem[{d'Aspremont et~al.(2008)d'Aspremont, Banerjee and
  El~Ghaoui}]{daspremontetal08}
\textsc{d'Aspremont, A.}, \textsc{Banerjee, O.} and \textsc{El~Ghaoui, L.}
  (2008).
\newblock First-order methods for sparse covariance selection.
\newblock \textit{SIAM J. Matrix Anal. Appl.} \textbf{30} 56--66.

\bibitem[{Drton and Perlman(2004)}]{drtonetperlman04}
\textsc{Drton, M.} and \textsc{Perlman, M.~D.} (2004).
\newblock Model selection for {G}aussian concentration graphs.
\newblock \textit{Biometrika} \textbf{91} 591--602.

\bibitem[{Fan and Li(2001)}]{fanetli01}
\textsc{Fan, J.} and \textsc{Li, R.} (2001).
\newblock Variable selection via nonconcave penalized likelihood and its oracle
  properties.
\newblock \textit{J. Amer. Statist. Assoc.} \textbf{96} 1348--1360.

\bibitem[{Georgii(1988)}]{georgii88}
\textsc{Georgii, H.-O.} (1988).
\newblock \textit{Gibbs measures and phase transitions}, vol.~9 of \textit{de
  Gruyter Studies in Mathematics}.
\newblock Walter de Gruyter \& Co., Berlin.

\bibitem[{Guo et~al.(2010)Guo, Levina, Michailidis and Zhu}]{guoetal10}
\textsc{Guo, J.}, \textsc{Levina, E.}, \textsc{Michailidis, G.} and
  \textsc{Zhu, J.} (2010).
\newblock Joint structure estimation for categorical markov networks.
\newblock Tech. rep., Univ. of Michigan.

\bibitem[{Hess(1996)}]{hess96}
\textsc{Hess, C.} (1996).
\newblock Epi-convergence of sequences of normal integrands and strong
  consistency of the maximum likelihood estimator.
\newblock \textit{Annals of Statistics} \textbf{24} 1298--1315.

\bibitem[{H{\"o}fling and Tibshirani(2009)}]{hoefling09}
\textsc{H{\"o}fling, H.} and \textsc{Tibshirani, R.} (2009).
\newblock Estimation of sparse binary pairwise {M}arkov networks using
  pseudo-likelihoods.
\newblock \textit{J. Mach. Learn. Res.} \textbf{10} 883--906.

\bibitem[{Lam and Fan(2009)}]{lametfan09}
\textsc{Lam, C.} and \textsc{Fan, J.} (2009).
\newblock Sparsistency and rates of convergence in large covariance matrix
  estimation.
\newblock \textit{Ann. Statist.} \textbf{37} 4254--4278.

\bibitem[{Meinshausen and Buhlmann(2006)}]{meinshausen06}
\textsc{Meinshausen, N.} and \textsc{Buhlmann, P.} (2006).
\newblock High-dimensional graphs with the lasso.
\newblock \textit{Annals of Stat.} \textbf{34} 1436--1462.

\bibitem[{Meinshausen and Yu(2009)}]{meinshausenetal09}
\textsc{Meinshausen, N.} and \textsc{Yu, B.} (2009).
\newblock Lasso-type recovery of sparse representations for high-dimensional
  data.
\newblock \textit{Ann. Statist.} \textbf{37} 246--270.

\bibitem[{Ravikumar et~al.(2010)Ravikumar, Wainwright and
  Lafferty}]{ravikumaretal10}
\textsc{Ravikumar, P.}, \textsc{Wainwright, M.~J.} and \textsc{Lafferty, J.~D.}
  (2010).
\newblock High-dimensional {I}sing model selection using {$\ell_1$}-regularized
  logistic regression.
\newblock \textit{Ann. Statist.} \textbf{38} 1287--1319.

\bibitem[{Rothman et~al.(2008)Rothman, Bickel, Levina and Zhu}]{rothmanetal08}
\textsc{Rothman, A.~J.}, \textsc{Bickel, P.~J.}, \textsc{Levina, E.} and
  \textsc{Zhu, J.} (2008).
\newblock Sparse permutation invariant covariance estimation.
\newblock \textit{Electron. J. Stat.} \textbf{2} 494--515.

\bibitem[{van~der Vaart and Wellner(1996)}]{vaartetwellner96}
\textsc{van~der Vaart, A.~W.} and \textsc{Wellner, J.~A.} (1996).
\newblock \textit{Weak convergence and empirical processes}.
\newblock Springer series in Statistics, Springer, New York.

\bibitem[{Vershynin(2009)}]{roman09}
\textsc{Vershynin, R.} (2009).
\newblock On the role of sparsity in compressed sensing and random matrix
  theory.
\newblock Tech. rep., ArXiv:0908.0257v1.

\bibitem[{Xue et~al.(2010)Xue, Zou and Cai}]{xueetal10}
\textsc{Xue, L.}, \textsc{Zou, H.} and \textsc{Cai, T.} (2010).
\newblock Non-concave penalized composite likelihood estimation of sparse ising
  models.
\newblock Tech. rep., Univ. of Minnesota.

\bibitem[{Yuan and Lin(2007)}]{yuanetlin07}
\textsc{Yuan, M.} and \textsc{Lin, Y.} (2007).
\newblock Model selection and estimation in the {G}aussian graphical model.
\newblock \textit{Biometrika} \textbf{94} 19--35.

\end{thebibliography}

\end{document}